\documentclass[11pt]{amsart}

\usepackage{amsaddr}
\usepackage{amscd}
\usepackage{amsfonts}
\usepackage{amsmath}
\usepackage{amssymb}
\usepackage{amsthm}
\usepackage{bbm} 
\usepackage{cite}
\usepackage{color}
\usepackage{cases}
\usepackage{latexsym}
\usepackage{mathtools}
\usepackage{microtype}
\usepackage[active]{srcltx}
\usepackage[UKenglish]{babel}
\usepackage{url}

\usepackage{hyperref}
\usepackage{enumitem}
\usepackage{cleveref}

\usepackage[charter]{mathdesign}

\usepackage[mathscr]{euscript}
\usepackage{calrsfs}
\DeclareMathAlphabet{\mathcalalt}{OMS}{cmsy}{m}{n}

\newcommand{\poalgfont}{\mathcalalt}
\newcommand{\bbfont}{\mathbbm}
\newcommand{\lebfont}{\mathcal}
\newcommand{\opstrufont}{\mathcalalt}
\newcommand{\borelfont}{\mathscr}

\allowdisplaybreaks
\numberwithin{equation}{section}

\newcommand{\upc}{{\mathrm{c}}}
\newcommand{\upd}{{\mathrm{d}}}
\newcommand{\upC}{{\mathrm{C}}}

\newcommand{\NN}{{\bbfont N}}
\newcommand{\RR}{{\bbfont R}}

\newcommand{\ulp}{{\textup{(}}}
\newcommand{\urp}{{\textup{)}}}
\newcommand{\uppars}[1]{\ulp #1\urp}

\newcommand{\abs}[1]{{\lvert #1 \rvert}}
\newcommand{\norm}[1]{{\lVert #1 \rVert}}

\newcommand{\mc}{mono\-tone com\-plete}
\newcommand{\smc}{$\sigma$-mono\-tone com\-plete}
\newcommand{\Dc}{De\-de\-kind com\-plete}
\newcommand{\sDc}{$\sigma$-De\-de\-kind com\-plete}
\newcommand{\wir}{weakly inner regular}

\newcommand{\SOT}{{\ensuremath{\mathrm{SOT}}}}

\newcommand{\comp}{{\upc}}

\newcommand{\indicator}[1]{\chi_{#1}}
\newcommand{\zerofunction}{\textbf{0}}
\newcommand{\onefunction}{\textbf{1}}

\newcommand{\pos}[1]{{#1^+}}
\newcommand{\negt}[1]{{#1^-}}

\newcommand{\seq}[1]{\{{#1}_n\}_{n=1}^{\infty}}
\newcommand{\net}[1]{\{{#1}_\lambda\}_{\lambda\in \Lambda}}

\newcommand{\largest}{\infty}

\newcommand{\sa}{{\mathrm{sa}}}

\newcommand{\supp}{\mathrm{supp}\,}

\newcommand{\f}[1]{(#1)} 
\newcommand{\lrf}[1]{\left(#1\right)}

\newcommand{\inp}[1]{\langle #1 \rangle}
\newcommand{\lrinp}[1]{\left\langle #1\right\rangle}

\newcommand{\pset}{X}
\newcommand{\pt}{x}
\newcommand{\ts}{X}

\newcommand{\os}{E}

\newcommand{\posos}{\pos{\os}}
\newcommand{\osext}{\overline{\os}}
\newcommand{\pososext}{\overline{\posos}}

\newcommand{\posR}{\pos{\RR}}
\newcommand{\Rext}{\overline{\RR}}
\newcommand{\posRext}{\overline{\pos{\RR}}}

\newcommand{\hilbert}{H}
\newcommand{\jbw}{\poalgfont{M}}

\newcommand{\posmap}{\pi}
\newcommand{\posmapxp}{\posmap_{x^\prime}}

\newcommand{\idop}{I}

\newcommand{\odual}[1]{{#1^{\thicksim}}}
\newcommand{\ndual}[1]{{#1^{\ast}}}
\newcommand{\ocdual}[1]{{#1_{\mathrm {oc}}^{\thicksim}}}
\newcommand{\socdual}[1]{{#1_{\sigma\mathrm {oc}}^{\thicksim}}}

\newcommand{\odualos}{\odual{\os}}
\newcommand{\ndualos}{\ndual{\os}}
\newcommand{\ocdualos}{\ocdual{\os}}
\newcommand{\socdualos}{\socdual{\os}}

\newcommand{\bounded}{{\opstrufont B}}
\newcommand{\linear}{{\opstrufont L}}

\newcommand{\boundedh}{\bounded (\hilbert)}

\newcommand{\ocontinuous}{\linear_{\mathrm{oc}}}
\newcommand{\socontinuous}{\linear_{\sigma\mathrm{oc}}}
\newcommand{\linearop}[1]{\linear(#1)}

\newcommand{\regularop}[1]{\linear_{\mathrm r}(#1)}
\newcommand{\ocontop}[1]{\ocontinuous(#1)}
\newcommand{\socontop}[1]{\socontinuous(#1)}

\newcommand{\alg}{\Omega}

\newcommand{\borel}{\borelfont B}
\newcommand{\rcompacto}{\Upsilon}

\newcommand{\mss}{\Delta}
\newcommand{\msstwo}{\Gamma}

\newcommand{\ms}{(\pset,\alg)}

\newcommand{\npm}{\mu}
\newcommand{\pom}{\npm^\ast}

\newcommand{\npn}{\nu}

\newcommand{\npmxp}{\npm_{x^\prime}}
\newcommand{\npnxp}{\npn_{x^\prime}}

\newcommand{\spaceofmeasuresletter}{{\mathrm M}}

\newcommand{\posextmeasts}{{\spaceofmeasuresletter(\ts,\borel;\pososext)}}
\newcommand{\posextBmeas}{{\spaceofmeasuresletter_{\mathrm {B}}(\ts,\borel;\pososext)}}
\newcommand{\posrBmeas}{{\spaceofmeasuresletter_{\mathrm {rB}}(\ts,\borel;\posos)}}
\newcommand{\posextrBmeas}{{\spaceofmeasuresletter_{\mathrm {rB}}(\ts,\borel;\pososext)}}
\newcommand{\posqrBmeas}{{\spaceofmeasuresletter_{\mathrm {qrB}}(\ts,\borel;\posos)}}
\newcommand{\posextqrBmeas}{{\spaceofmeasuresletter_{\mathrm {qrB}}(\ts,\borel;\pososext)}}

\newcommand{\opint}[1]{I_{#1}}
\newcommand{\opintm}{\opint{\npm}}

\newcommand{\di}[1]{\,\upd #1}
\newcommand{\orderintegral}[3]{{\int_{#1}^{\mathrm{o}}\! {#2}\di {#3}}}
\newcommand{\ointm}[1]{\orderintegral{\pset}{#1}{\npm}}

\newcommand{\cont}[1]{{\upC}(#1)}
\newcommand{\conto}[1]{\upC_0(#1)}
\newcommand{\contc}[1]{\upC_{\upc}(#1)}
\newcommand{\contts}{\cont{\ts}}
\newcommand{\contots}{\conto{\ts}}
\newcommand{\contcts}{\contc{\ts}}

\newcommand{\integrablefun}{{{\lebfont L}^1(\pset,\alg,\npm;\RR)}}

\newcommand{\boundedmeasfun}{{{\lebfont{B}(\pset,\alg;\RR)}}}

\newcommand{\ellone}{{{\mathrm L}^1(\pset,\alg,\npm;\RR)}}

\newcommand{\elemfunts}{{{\lebfont E}(\ts,\borel;\posR)}}
\newcommand{\integrablefunts}{{{\lebfont L}^1(\ts,\borel,\npm;\RR)}}
\newcommand{\boundedmeasfunts}{{{\lebfont B (\ts,\borel;\RR)}}}

\newtheorem{theorem}{Theorem}[section]
\newtheorem{proposition}[theorem]{Proposition}
\newtheorem{lemma}[theorem]{Lemma}
\newtheorem{corollary}[theorem]{Corollary}

\theoremstyle{definition}
\newtheorem{definition}[theorem]{Definition}

\theoremstyle{remark}
\newtheorem{remark}[theorem]{Remark}

\newtheorem{examples}[theorem]{Examples}
\setlist[enumerate,1]{label=\textup{(\arabic*)},ref=\textup{(\arabic*)}}
\setlist[enumerate,2]{label=\textup{(\alph*)},ref=\textup{(\alph*)}}
\setlist[enumerate,3]{label=\textup{(\roman*)},ref=\textup{(\roman*)}}
\setlist[enumerate,4]{label=\textup{(\Alph*)},ref=\textup{(\Alph*)}}

\newlist{enumerate_alpha}{enumerate}{1}
\setlist[enumerate_alpha,1]{label=\textup{(\alph*)},ref=\textup{(\alph*)}}

\crefname{theorem}{Theorem}{Theorems}
\crefname{proposition}{Proposition}{Propositions}
\crefname{lemma}{Lemma}{Lemmas}
\crefname{corollary}{Corollary}{Corollaries}
\crefname{definition}{Definition}{Definitions}
\crefname{example}{Example}{Examples}
\crefname{examples}{Examples}{Examples}
\crefname{remark}{Remark}{Remarks}
\crefname{equation}{equation}{equations}

\crefname{section}{Section}{Sections}
\crefname{subsection}{Section}{Sections}
\crefname{subsubsection}{Section}{Sections}

\begin{document}

\title[Riesz representation theorems for positive linear operators]{Riesz representation theorems for positive linear operators}

\author{Marcel de Jeu}
\address[Marcel de Jeu]{Mathematical Institute, Leiden University, P.O.\ Box 9512, 2300 RA Leiden, The Netherlands\\
	and\\
	Department of Mathematics and Applied Mathematics, University of Pretoria, Corner of Lynnwood Road and Roper Street, Hatfield 0083, Pretoria,
	South Africa
	}
\email[Marcel de Jeu]{mdejeu@math.leidenuniv.nl}

\author{Xingni Jiang}
\address[Xingni Jiang]{College of Mathematics, Sichuan University, No.\ 24, South Section, First Ring Road, Chengdu, P.R.\ China}
\email[Xingni Jiang]{x.jiang@scu.edu.cn}


\begin{abstract}
 We generalise the Riesz representation theorems for positive linear functionals on $\mathrm{C}_{\mathrm c}(X)$ and $\mathrm{C}_{\mathrm 0}(X)$, where $X$ is a locally compact Hausdorff space, to positive linear operators from these spaces into a partially ordered vector space $E$. The representing measures are defined on the Borel $\sigma$-algebra of $X$ and take their values in the extended positive cone of $E$. The corresponding integrals are order integrals. We give explicit formulas for the values of the representing measures at open and at compact subsets of $X$.\\
 Results are included where the space $E$ need not be a vector lattice, nor a normed space.  Representing measures exist, for example, for positive linear operators into Banach lattices with order continuous norms, into the regular operators on KB-spaces, into the self-adjoint linear operators on complex Hilbert spaces, and into JBW-algebras.
\end{abstract}

\maketitle

\section{Introduction and overview}\label{2_sec:introduction}

\noindent Let $\ts$ be a locally compact Hausdorff space, and let $\posmap:\contcts\to\RR$ be a positive linear functional. The Riesz representation theorem asserts that there is a unique regular Borel measure on the Borel $\sigma$-algebra of $\ts$, such that
\begin{equation}\label{2_eq:classical_equation}
	\posmap(f)=\int_\ts \! f\di{\npm}
	\end{equation}
	for all $f$ in $\contcts$.

In this paper, we establish analogous representation theorems for positive linear operators $\posmap:\contcts\to\os$ and $\posmap:\contots\to\os$, where $\os$ is a (suitable) partially ordered vector space, while giving explicit formulas for the measures of open and of compact subsets of $\ts$.\footnote{In the course of the present paper, its prequel \cite{de_jeu_jiang:2021a}, and its sequels \cite{de_jeu_jiang:2021c,de_jeu_jiang:2021d}, we shall encounter maps with $\contcts$ or $\contots$ as domains that are sometimes positive linear operators, sometimes vector lattice homomorphisms, and sometimes positive algebra homomorphisms. For each of these contexts, a canonical symbol for such maps could be chosen. However, since our results for these contexts are related, we have chosen to use the same symbol $\posmap$ throughout, thus keeping the notation as uniform as possible.} Results are included where the space $\os$ need not be a vector lattice, nor a normed space. As will become clear, representing measures exist for, e.g., positive linear operators into Banach lattices with order continuous norms, into the regular operators on KB-spaces, into the self-adjoint operators on complex Hilbert spaces, and into JBW-algebras. When $\os=\RR$, the results specialise to the classical Riesz representation theorems.

In the general setting of the present paper, the measure $\npm$ in \cref{2_eq:classical_equation} takes its values in the (extended) positive cone of $\os$, and the integral in question is an order integral. Such measures and their order integrals are the subject of \cite{de_jeu_jiang:2021a}, which extends earlier work by Wright.  The results for the order integral in \cite{de_jeu_jiang:2021a} are fairly complete and also include convergence theorems that can be of use in applications, with the existence theorems in the present paper as a starting point.

A possible multiplicativity of the linear operator $\posmap$ is not an issue in the current paper: being positive and linear is enough. In the sequel \cite{de_jeu_jiang:2021c}, we shall consider positive algebra homomorphisms from $\contcts$ or $\contots$ into partially ordered algebras. The representing measures from the current paper can then be shown to be \emph{spectral} measures that take values in the algebras. It will be seen in \cite{de_jeu_jiang:2021c} that the ensuing existence theorems for abstract spectral measures immediately imply the classical ones for representations of (the complexification of) $\contots$ on complex Hilbert spaces, and for positive representations of $\contots$ on KB-spaces in \cite{de_jeu_ruoff:2016}. The up-down theorems that are established in \cite{de_jeu_jiang:2021c} for general partially ordered algebras yield the familiar results for Hilbert spaces as special cases.  In \cite{de_jeu_jiang:2021d}, which is another sequel to the present paper, we shall be concerned with representation theorems for vector lattices (resp.\ Banach lattices) of regular operators from $\contcts$ and $\contots$ into \Dc\ vector lattices (resp.\ Banach lattices with order continuous norms) in the spirit of \cite[Theorem~38.7]{aliprantis_burkinshaw_PRINCIPLES_OF_REAL_ANALYSIS_THIRD_EDITION:1998}.

We shall discuss the relation between the present paper and existing representation theorems for positive linear operators in the literature at the end of this introduction. There appears to be no previous work in the vein of the sequels \cite{de_jeu_jiang:2021d} or \cite{de_jeu_jiang:2021c} to the present paper.

\medskip

\noindent This paper is organised as follows.

\cref{2_sec:preliminaries} contains the necessary prerequisites from \cite{de_jeu_jiang:2021a}, including those on measures and the order integral, and can serve as a summary thereof.

In the preparatory \cref{2_sec:measures_on_locally_compact_hausdorff_spaces}, we introduce various types of regularity of measures on the Borel $\sigma$-algebra of a locally compact Hausdorff space that take their values in the extended positive cones of  partially ordered vector spaces. A few auxiliary results for the remainder of the paper are also established,

\cref{2_sec:riesz_representation_theorems_for_contcts_normed_case} contains the proof of a representation theorem (see \cref{2_res:riesz_representation_theorem_for_contcts_normed_case}) for positive linear operators $\posmap:\contcts\to\os$, where $\os$ is (for all practical purposes) a Banach lattice with an order continuous norm. It is one of the essential ingredients for the sequel \cite{de_jeu_jiang:2021d}. The possibly infinite representing measure in it is always regular.

Since spaces of operators will only rarely satisfy the conditions on the space $\os$ in \cref{2_sec:riesz_representation_theorems_for_contcts_normed_case}, other results are needed that \emph{do} apply (at least) when $\os$ is a space of operators. Such results are to be found in \cref{2_sec:riesz_representation_theorems_for_contcts_normal_case}, where representation theorems (see \cref{2_res:riesz_representation_theorem_for_contcts_finite_normal_case,2_res:riesz_representation_theorem_for_contcts_normal_case}) are established when the codomain of the positive linear operator $\posmap:\contcts\to\os$ is a \mc\ and normal space. The class of such spaces $\os$ is quite varied. We refer to \cref{2_ex:combination_result_for_normality_and_monotone_completeness} and also to \cite[Section~3]{de_jeu_jiang:2021a} for examples, some of which were already mentioned above. When $\os$ consists of the regular operators on a Banach lattice with an order continuous norm, or of the self-adjoint operators in a strongly closed complex linear subspace of the bounded linear operators on a complex Hilbert space, then the representing measure in \cref{2_res:riesz_representation_theorem_for_contcts_finite_normal_case} has the familiar property of being $\sigma$-additive in the strong operator topology. The representing measures in \cref{2_res:riesz_representation_theorem_for_contcts_finite_normal_case} are finite (by assumption) and regular. Those in \cref{2_res:riesz_representation_theorem_for_contcts_normal_case} can be infinite, but regularity need then hold only locally. In \cref{3_rem:comparison}, we compare the applicability of the main representation theorems (\cref{2_res:riesz_representation_theorem_for_contcts_normed_case,2_res:riesz_representation_theorem_for_contcts_finite_normal_case,2_res:riesz_representation_theorem_for_contcts_normal_case}) for positive linear operators $\posmap:\contots\to\os$.

The final \cref{2_sec:riesz_representation_theorems_for_contots} is concerned with positive linear operators $\posmap:\contots\to\os$. The domain $\contots$ is now a Banach lattice, and we can then exploit automatic continuity to derive representation theorems for such linear operators from those for their restrictions to $\contcts$. The representing measures thus obtained are all finite. \cref{2_res:riesz_representation_theorem_for_contots_kb_case} covers the case where $\os$ is a KB-space, and \cref{2_res:riesz_representation_theorem_for_contots_normal_case} applies to the larger class of quasi-perfect partially ordered vector spaces. In \cref{2_res:riesz_representation_theorem_for_contots_operators_on_kb_space,2_res:riesz_representation_theorem_for_contots_operators_on_hilbert_space}, the space $\os$ consists of the regular operators on a KB-space and of the self-adjoint operators in a strongly closed complex linear subspace of the bounded linear operators on a complex Hilbert space, respectively. In both cases, the representing measures are strongly $\sigma$-additive again. Furthermore, a number of \SOT-closed subspaces of $\os$ that are naturally associated with $\posmap$ can be seen to coincide. As a consequence, the representing measure takes its values in the coinciding bicommutants of $\posmap(\contcts)$ and $\posmap(\contots)$.

\medskip

\noindent We now give an overview of earlier work that we are aware of on Riesz representation theorems for positive linear operators from spaces of continuous functions into various types of partially ordered vector spaces. The space $\ts$ is always a locally compact Hausdorff space, unless otherwise stated. As will become clear, one can hardly speak of `the' Riesz representation theorem for positive linear operators from $\contcts$ or $\contots$ into partially ordered vector spaces.

To start with, there is the seminal paper by Wright \cite{wright:1969}. It has a representation theorem for a positive linear operator from $\cont{\ts}$, where $\ts$ is compact, into a Stone algebra. A second paper \cite{wright:1971b} by the same author covers the case of a positive linear operator from $\contts$, where $\ts$ is compact, into a \sDc\ vector lattice. A third paper \cite{wright:1971a} contains a representation theorem for a positive linear operator from $\contcts$ into a \Dc\ vector lattice. In a fourth \cite{wright:1972}, the existence is established of a representing Baire measure for a positive linear operator from $\contts$, where $\ts$ is compact, into a \smc\ partially ordered vector space. When comparing the latter result to our \cref{2_res:riesz_representation_theorem_for_contcts_finite_normal_case}, both have their strong points. Wright's conditions on $\os$ are more lenient, but the space $\ts$ has to be compact, which is essential to the proofs. \cref{2_res:riesz_representation_theorem_for_contcts_finite_normal_case}, on the other hand, is valid for locally compact spaces, and also gives explicit formulas for the measure of open and of compact subsets.

In \cite{khurana:1976}, one of Khurana's papers in this direction, it is shown that a positive linear operator from $\contc{\ts}$ into a \mc\ partially ordered vector space can be extended to a $\sigma$-order continuous linear operator defined on the bounded Borel functions with compact support; this goes in the direction of a Riesz representation theorem. In a second paper \cite{khurana:1978}, a representation theorem is established, under certain additional conditions, for a positive linear operator from the bounded continuous functions on a completely regular Hausdorff space into a \mc\ par\-tial\-ly ordered vector space. In \cite{khurana:2008a}, he proves a representation theorem for a positive linear operator from the continuous functions on a completely regular $\mathrm{T}_1$-space into a Stone algebra.

In \cite{lipecki:1987}, Lipecki shows the existence of a representing finitely additive measure for positive linear operators from the bounded continuous functions on an arbitrary topological space into a \mc\ partially ordered vector space; there is a second result for the bounded continuous functions when the topological space is normal.

In \cite{coquand:2004}, Coquand gives a new proof of Wright's result in \cite{wright:1972}.

The results cited above also include regularity properties of the representing measures, but not always the same properties. In cases where several results apply, it is, therefore, not clear whether the representing measures are necessarily equal.

The results in the present paper have the advantage that the space $\ts$ need not be compact and that they apply, amongst others, when $E$ is a \mc\ normal space.  This is a fairly large class of spaces, containing many spaces that are not vector lattices. Moreover, explicit formulas for the measures of open and of compact subsets of $\ts$ are given. We are not aware of a similar combination of a reasonably wide range of applicability and concrete formulas in the literature on partially ordered vector spaces. The usefulness of this combination can be seen in, e.g, \cref{2_res:riesz_representation_theorem_for_contots_operators_on_kb_space,2_res:riesz_representation_theorem_for_contots_operators_on_hilbert_space}, which, though in two rather different contexts, are both virtually immediate consequences of one underlying general result. It will become even more pronounced in the sequel \cite{de_jeu_jiang:2021c} when considering positive algebra homomorphisms from $\contcts$ and $\contots$ into \mc\ normal partially ordered algebras with monotone continuous multiplications.

\section{Preliminaries}\label{2_sec:preliminaries}

\noindent In this section, we collect some conventions, definition, notations, and preparatory results that will be used in the sequel.

\medskip

\noindent The indicator function of a subset $S$ of a set $\pset$ is denoted by $\indicator{S}$. We shall also write $\zerofunction$ for $\indicator{\emptyset}$ and $\onefunction$ for $\indicator{\pset}$. When $\ts$ is a topological space, then we let $\contcts$, resp.\ $\contots$, denote the continuous real-valued functions on $\ts$ that have compact support, resp.\ vanish at infinity. 

\subsection{Partially ordered vector spaces}\label{2_subsec:partially_ordered_vector_spaces}

\noindent All vector spaces we shall consider are over the real numbers, unless otherwise indicated. An operator between two vector spaces and a functional are always supposed to be linear, but\textemdash when this notion is applicable\textemdash need not be bounded. We do not require that the positive cone $\pos{\os}$ of a partially ordered vector space $\os$ be generating. Equivalently, we do not require that $\os$ be directed. We do require, however, that $\posos$ be proper, i.e., that $\posos\cap(-\posos)=\{0\}$. All vector lattices are supposed to be Archimedean. 

\begin{definition}\label{2_def:order_completeness} A partially ordered vector space $\os$ is called
	\begin{enumerate}
		\item\label{2_part:order_completeness_1}
		\emph{\smc} if every increasing sequence $\seq{x}$ in $\os$ that is bounded from above has a supremum in $\os$;
		\item\label{2_part:order_completeness_2}
		\emph{\mc} if every increasing net $\net{x}$ in $\os$ that is bounded from above has a supremum in $\os$;
		\item\label{2_part:order_completeness_3}
		\emph{\sDc} if every non-empty at most countably infinite subset $S$ of $\os$ that is bounded from above has a supremum in $\os$;
		\item\label{2_part:order_completeness_4}
		\emph{\Dc} if every non-empty subset $S$ that is bounded from above has a supremum  in $\os$.
	\end{enumerate}
\end{definition}

We shall employ the usual notation in which $x_\lambda\downarrow$ means that $\net{x}$ is a decreasing net, and in which $x_\lambda\downarrow x$ means that $\net{x}$ is a decreasing net with infimum $x$. The notations $x_\lambda\uparrow$ and $x_\lambda\uparrow x$ are similarly defined.

It was observed in \cite[Lemma~1.1]{wright:1972} that every \smc\ partially ordered vector space $\os$ (and then also every \mc, \sDc, or \Dc\  partially order vector space) is Archimedean, i.e., $\bigwedge\{\varepsilon x : \varepsilon >0\}=0$ for all $x\in\pos{\os}$. We shall use this a number of times.

\medskip

\noindent Vector spaces of operators between partially ordered vector spaces can inherit completeness properties from the codomains. In order to formulate this, we first introduce some notation and terminology.

When $\os$ and $F$ are vector spaces, then $\linearop{\os,F}$ denotes the vector space of operators from $\os$ into $F$. An operator $T\in\linearop{\os,F}$ between two partially ordered vector spaces is \emph{positive} if $T(\pos{\os})\subseteq \pos{F}$, and \emph{regular} if it is the difference of two positive operators.  The regular operators from $\os$ into $F$ form a vector space that is denoted by $\regularop{\os,F}$. When $\pos{\os}$ is directed, then every linear subspace of $\linearop{\os,F}$ that contains $\regularop{\os,F}$ is naturally partially ordered with the regular positive operators from $\os$ into $F$, denoted by $\pos{\regularop{\os,F}}$, as its positive cone.
We shall write $\linearop{\os}$ for $\linearop{\os,\os}$,  $\regularop{\os}$ for $\regularop {\os,\os}$, and $\odualos$ for $\regularop{\os,\RR}$. When $\os$ is a Banach lattice, then $\odualos$ coincides with the norm dual $\ndualos$ of $\os$.

We can now state how completeness is hereditary. When $\os$ is a directed partially ordered vector space, and $F$ is a partially ordered vector space that is \mc\  \uppars{resp.\ \smc}, then any linear subspace of $\linearop{\os,F}$ containing $\regularop{\os,F}$  is \mc\  \uppars{resp.\ \smc}; see \cite[Proposition~3.1]{de_jeu_jiang:2021a}.

\medskip

\noindent The \mc\ partially ordered vector spaces that are also normal will play an
important part in this paper. We now proceed to define the latter notion.

\begin{definition}\label{2_def:order_continuity} Let $\os$ and $F$ be partially ordered vector spaces, and let $T:\os\to F$ be a positive operator. Then $T$ is called \emph{order continuous} (resp.\ \emph{$\sigma$-order continuous}) if $Tx_\lambda\downarrow 0$ in $F$ whenever $x_\lambda\downarrow 0$ in $\os$ (resp.\ if $Tx_n\downarrow 0$ in $F$ whenever $x_n\downarrow 0$ in $\os$).  A general operator in $\regularop{\os,F}$ is called order continuous (resp.\ $\sigma$-order continuous) if it is the difference of two positive order continuous operators. We let $\ocontop{\os,F}$ (resp.\ $\socontop{\os,F}$) denote the order continuous (resp.\ $\sigma$-order continuous) operators from $\os$ into $F$; we shall write $\ocdualos$ for $\ocontop{\os,\mathbb R}$ and $\socdualos$ for $\socontop{\os,\RR}$.\footnote{When $\os$ and $F$ are vector lattices, where $F$ is \Dc, then the notions of order continuous and $\sigma$-order continuous operators in \cref{2_def:order_continuity} agree with the usual ones in the literature as on \cite[p.~123]{zaanen_RIESZ_SPACES_VOLUME_II:1983}; see \cite[Remark~3.5]{de_jeu_jiang:2021a} for this.}
\end{definition}

It is easy to see that $\ocontop{\os,F}$ and $\socontop{\os,F}$ are linear subspaces of $\regularop{\os,F}$. When $\os$ is directed, then they are partially ordered vector spaces with the positive order continuous operators (resp.\ the positive  $\sigma$-order continuous operators) from $\os$ into $F$ as positive cones, which are generating by definition.

\begin{definition}\label{2_def:normal_space}
Let $\os$ be a partially ordered vector space. Then $\os$ is called \emph{normal} when, for $x\in \os$, $\f{x,x^\prime}\geq 0$ for all $x^\prime\in\pos{(\ocdualos)}$ if and only if $x\in \pos{\os}$. We say that $\os$ is \emph{$\sigma$-normal} when, for $x\in \os$, $\f{x,x^\prime}\geq 0$ for all $x^\prime\in\pos{(\socdualos)}$ if and only if $x\in \pos{\os}$.\footnote{
	When $\os$ is a vector lattice, then the notion of normality in \cref{2_def:normal_space} coincides with the usual one in the literature (see \cite[p.~21]{abramovich_aliprantis_INVITATION_TO_OPERATOR_THEORY:2002}, for example) that $\ocdualos$ separates the points of $\os$; see \cite[Lemma~3.7]{de_jeu_jiang:2021a} for this. It also follows from \cite[Lemma~3.7]{de_jeu_jiang:2021a} that a vector lattice $\os$ is $\sigma$-normal as in \cref{2_def:normal_space} if and only if $\socdualos$ separates the points of $\os$.}
\end{definition}

The importance of normality for our work lies in the following result (see \cite[Proposition~3.8]{de_jeu_jiang:2021a}) that will be used quite a few times in the present paper.

\begin{proposition}\label{2_res:inf_and_sup_via_order_dual}
Let $\os$ be a normal partially ordered vector space. Suppose that $\net{x}$ is a net in $\os$, and that $x\in\os$.
\begin{enumerate}
\item If $x_\lambda\downarrow$, then $x_\lambda\downarrow x$ if and only if $\f{x,x^\prime}=\inf_{\lambda\in\Lambda }\f{x_\lambda,x^\prime}$ for all $x^\prime\in\pos{(\ocdualos)}$.\label{2_part:inf}
\item If $x_\lambda\uparrow$, then $x_\lambda\uparrow x$ if and only if $\f{x,x^\prime}=\sup_{\lambda\in\Lambda }\f{x_\lambda,x^\prime}$ for all $x^\prime\in\pos{(\ocdualos)}$.\label{2_part:sup}
\end{enumerate}
\end{proposition}

In the presence of \mc ness, normality of a codomain is a hereditary property: when $\os$ is a directed partially ordered vector space, and $F$ is a \mc\ and normal partially ordered vector space, then any linear subspace of $\linearop{\os,F}$ containing $\regularop{\os,F}$ is \mc\ and normal; see \cite[Proposition~3.11]{de_jeu_jiang:2021a}.

Since both ($\sigma$-)\mc ness and normality are properties that are inherited by partially ordered vector spaces of operators, it is easy to construct examples of \mc\ and normal spaces once one has such a space to begin with.

\begin{examples}\label{2_ex:combination_result_for_normality_and_monotone_completeness}

In \cite[Section~3]{de_jeu_jiang:2021a} we have included a number of examples of partially ordered vector spaces that are \mc\ and normal. These include (but are not limited to):
\begin{enumerate}
	\item\label{2_part:combination_result_for_normality_and_monotone_completeness_1}
	Banach lattices with order continuous norms;
	\item\label{2_part:combination_result_for_normality_and_monotone_completeness_2}
	for partially ordered vector spaces $\os$ and $F$ such that $\os$ is directed and $F$ is \mc\ and normal: every linear subspace of $\linearop{\os,F}$ that contains $\regularop{\os,F}$;
	\item\label{2_part:combination_result_for_normality_and_monotone_completeness_3} as a special case of part~\ref{2_part:combination_result_for_normality_and_monotone_completeness_2}:
	the regular operators on a Banach lattice with an order continuous norm;
	\item\label{2_part:combination_result_for_normality_and_monotone_completeness_4}
	the real vector space that consists of the self-adjoint operators in a strongly closed complex linear subspace of the bounded operators on a complex Hilbert space;
	\item \label{2_part:combination_result_for_normality_and_monotone_completeness_5}
	JBW-algebras.\footnote{We shall use \cite{alfsen_shultz_STATE_SPACES_OF_OPERATOR_ALGEBRAS:2001,alfsen_shultz_GEOMETRY_OF_STATE_SPACES_OF_OPERATOR_ALGEBRAS:2003} references for JBW-algebras. In these books, a JBW-algebra is supposed to have an identity element; see \cite[Definitions~1.5 and~2.2]{alfsen_shultz_GEOMETRY_OF_STATE_SPACES_OF_OPERATOR_ALGEBRAS:2003}. In other sources, this is not supposed. However, as \cite[Lemma~4.1.7]{hanche-olsen_stormer_JORDAN_OPERATOR_ALGEBRAS:1984} shows, the existence of an identity element is, in fact, automatic.}
\end{enumerate}
\end{examples}
	
\subsection{Measures and order integrals}\label{2_subsec:measures_and_order_integrals}

\noindent In this section, we shall briefly summarise the relevant definitions and results for measures and order integrals from \cite[Sections~4 and~6]{de_jeu_jiang:2021a}. This extends earlier work by Wright and contains the usual Lebesgue integral as a special case. We refer to \cite{de_jeu_jiang:2021a} for a discussion of the relation with Wright's work.

We refrain from mentioning here in any detail the material on outer measures in \cite[Section~5]{de_jeu_jiang:2021a}. It is indispensable in the proof of \cref{2_res:riesz_representation_theorem_for_contcts_normed_case}, but it occurs only as an intermediate step and does not reappear.

Let $\os$ be a partially ordered vector space. As the case of the Lebesgue integral on the real line already shows, one cannot expect a representing measure for a positive operator $\posmap: \contcts\to\os$ to be finite. This is why we extend $\os$ by introducing the set  $\osext\coloneqq\os\cup\{\largest\}$ as a disjoint union, and extend the partial ordering from $\os$ to  $\osext$ by declaring that $x\leq \largest$ for all $x\in\osext$. Then $\pososext\coloneqq\posos\cup\{\largest\}$ is the set of positive elements of $\osext$. The elements of $\osext$ that are in $\os$ are called \emph{finite}. One makes $\osext$ into an abelian monoid by retaining the addition on $\os$ and defining $\largest+x\coloneqq\largest$ and $x+\largest\coloneqq\largest$ for all $x\in \osext$. Then $\pososext$ is a sub-monoid of $\osext$. When $x,y\in\osext$ are such that $x\leq y$, then $x+z\leq y+z$ for all $z\in\osext$.
We keep the action of $\RR$ on $\os$, and define $r\cdot\largest\coloneqq\largest$ for all $r\in\posR\setminus\{0\}$ and $0\cdot\largest\coloneqq 0$. Then $rx\leq ry$ for all $r,s\in\posR$ and $x,y\in\osext$ such that $r\leq s$ and $x\leq y$. Furthermore, $r(x+y)=rx+ry$, $(r+s)x=rx+sx$, and $(rs)x=r(sx)$ for all $r,s\in\posR$ and $x,y\in\pososext$. Such relations will be used in the sequel without further reference.

When being used to working with the extended real numbers, which are still linearly ordered, it may be relatively easy to make mistakes when arguing with the ordering of and the operations on $\os$ and $\osext$. It is for this reason that a fair number of technical tools have been collected in  \cite[Lemmas~2.3 to~2.5]{de_jeu_jiang:2021a} that will be used repeatedly in the present paper and its sequels. Whenever necessary, we shall be careful to indicate whether we are working in $\os$ or in $\osext$ when speaking of order bounds, suprema, or infima.

Now that we have the extended space $\osext$ available, it is possible to define the measures that concern us. It was Wright who first observed in \cite{wright:1969} that \cref{2_eq:sigma_additivity}, below, is the proper way to generalise the notion of $\sigma$-additivity from the real numbers to more general partially ordered vector spaces.

A \emph{measurable space} is a pair $\ms$, where $\pset$ is a set and $\alg$ is an algebra of subsets of $\pset$; i.e., $\alg$ is a non-empty collection of subsets of $\pset$ that is closed under taking complements and under taking finite unions. For the moment, we can still work with algebras of subsets, rather than $\sigma$-algebras.

\begin{definition}\label{2_def:positive_pososext_valued_measure}
Let $\ms$ be a measurable space, and let $\os$ be a \smc\  partially ordered vector space. A \emph{positive $\osext$-valued measure} is a map $\npm:\alg\rightarrow \pososext$ such that:
\begin{enumerate}
\item\label{2_part:pososext_valued_measure_1}
$\npm(\emptyset)=0$;
\item\label{2_part:pososext_valued_measure_2}
($\sigma$-additivity) whenever $\seq{\mss}$ is a pairwise disjoint sequence in $\alg$ with $\bigcup_{n=1}^\infty\mss_n\in\alg$, then
 \begin{equation}\label{2_eq:sigma_additivity}
 \npm\left(\bigcup_{n=1}^\infty\mss_n\right)=\bigvee_{N=1}^\infty\sum_{n=1}^N\npm(\mss_n)
 \end{equation}
 in $\osext$.
\end{enumerate}
\end{definition}

Since $\npm$ is $\pososext$-valued, it follows from the \smc ness of $\os$ that the supremum in the right hand side of \cref{2_eq:sigma_additivity} exists in $\osext$; see part~(1) of \cite[Lemma~2.5]{de_jeu_jiang:2021a}.  When $\npm(\pset)\in\posos$, then we say that $\npm$ is \emph{finite}, or that it is \emph{$\os$-valued}; when $\npm(\pset)=\infty$, then $\npm$ is said to be \emph{infinite}. As \cite[Section~4]{de_jeu_jiang:2021a} shows, a good number of the properties of positive $\Rext$-valued measures hold in the general case as well, including even the Borel--Cantelli lemmas (see \cite[Lemma~4.7]{de_jeu_jiang:2021a}).

Next, we introduce the integral with respect to a measure. From now on, we suppose that $\alg$ is a $\sigma$-algebra.

A measurable function $\varphi:\pset\to\posR$ is an \emph{elementary function} if it takes only finitely many (finite) values. It can be written as a finite sum $\varphi=\sum_{i=1}^n r_i\indicator{\mss_i}$ for some $r_1,\dotsc,r_n\in\posR$ and $\mss_1,\dots,\mss_n\in\alg$. Here the $r_i$ are all finite, but it is allowed that $\npm(\mss_i)=\infty$ for some of the $\mss_i$. We define its (order) integral, which is an element of $\pososext$, by setting
\[
\ointm{\varphi}\coloneqq\sum_{i=1}^n r_i\npm(\mss_i).
\]
This definition is independent of the above expression for $\varphi$ as a finite sum. When $f:\pset\to\posRext$  is measurable, we choose a sequence $\seq{\varphi}$ of elementary functions such that $\varphi_n\uparrow f$ pointwise in $\posRext$. We define the \emph{order integral} of $f$, which is an element of $\pososext$, by
\[
\ointm{f}\coloneqq \bigvee_{n=1}^\infty \ointm{\varphi_n}.
\]
Since $\ointm{\varphi_n}\uparrow$, the \smc ness of $\os$ guarantees that this supremum exists in $\pososext$. It is independent of the choice of the sequence $\seq{\varphi}$.

We let $\integrablefun$ denote the set of all (finite-valued) measurable functions $f:\pset\to\RR$ such that $\ointm{\abs{f}}$ is finite. By splitting a function into its positive and negative parts, the order integral is then defined on $\integrablefun$. We let $\boundedmeasfun$ denote the bounded measurable functions on $\pset$. For a finite measure $\npm$, $\boundedmeasfun\subseteq\integrablefun$.

In \cite[Section~6.2]{de_jeu_jiang:2021a}, the monotone convergence theorem for the order integral is established. When $\os$ is \sDc, then Fatou's lemma and the dominated convergence theorem are also valid.

The space  $\integrablefun$ is a \sDc\ vector lattice, and the order integral is a $\sigma$-order continuous positive operator from  $\integrablefun$ into $\os$; see \cite[Proposition~6.14]{de_jeu_jiang:2021a}. According to \cite[Theorem~6.17]{de_jeu_jiang:2021a}, the space $\ellone$, where elements of $\integrablefun$ have been identified when they agree $\npm$-almost everywhere, is likewise a \sDc\ vector lattice, and the order integral induces a strictly positive $\sigma$-order continuous operator $\opintm$ from $\ellone$ into $\os$. When $\os$ is monotone complete and has the countable sup property, then $\ellone$ is a Dedekind complete vector lattice with the countable sup property, and $\opintm$ is order continuous.\footnote{As in \cite[Section~6]{de_jeu_jiang:2021a}, we say that a partially ordered vector space $\os$ has the \emph{countable sup property} when, for every net $\net{x}\subseteq\posos$ and $x\in\pos{\os}$ such that $x_\lambda\uparrow x$, there exists an at most countably infinite set of indices $\{\lambda_n: n\geq 1 \}$ such that $x=\sup_{n\geq 1} x_{\lambda_n}$. In this case, there also always exist indices $\lambda_1\leq\lambda_2\leq\dotsb$ such that $x_{\lambda_n}\uparrow x$. For vector lattices, this definition coincides with the usual one as can be found on, e.g., \cite[p.~103]{aliprantis_burkinshaw_POSITIVE_OPERATORS_SPRINGER_REPRINT:2006}.}

\begin{remark}\label{2_rem:comparison_with_vector_measures_and_sot_sigma_additivity}\quad
	\begin{enumerate}
		\item\label{2_part:comparison_with_vector_measures_and_sot_sigma_additivity_1}
	When $E$ is a partially ordered Banach space with a closed positive cone, then every positive vector measure is a measure in the sense of \cref{2_def:positive_pososext_valued_measure}, but not conversely. Even when a measure falls into both categories, the domain of the order integral can properly contain that of any reasonably defined integral with respect to the vector measure using Banach space methods.
	We refer to \cite[Section~7]{de_jeu_jiang:2021a} for this discussion.
	\item\label{2_part:comparison_with_vector_measures_and_sot_sigma_additivity_2}
	When $\os$ consists of the regular operators on a Banach lattice with an order continuous norm, or when $\os$ consists of the self-adjoint operators in a strongly closed complex linear subspace of the bounded operator on a complex Hilbert space, then the finite measures in the sense of \cref{2_def:positive_pososext_valued_measure} are precisely the set maps $\npm:\alg\to\pos{E}$ with $\npm(\emptyset)=0$ that are $\sigma$-additive with respect to the strong operator topology on $\os$; see \cite[Lemmas~4.2 and~4.3]{de_jeu_jiang:2021a}.
	\end{enumerate}
\end{remark}

\section{Measures on locally compact Hausdorff spaces}\label{2_sec:measures_on_locally_compact_hausdorff_spaces}

\noindent There may be several measures on the Borel $\sigma$-algebra of a locally compact Hausdorff space $\ts$ that represent a given positive functional $\posmap: \contcts\to\RR$, but there is only one that is a regular Borel measure. In our vector-valued context, we shall have similar results. In this section, we shall define and investigate the pertinent regularity properties. The general situation is somewhat more involved than the real case.

\medskip

\noindent When $\ts$ is a locally compact Hausdorff space, then we let $\borel$ denote its Borel $\sigma$-algebra, i.e., $\borel$ is the $\sigma$-algebra that is generated by the open subsets of $\ts$.

When $\os$ is a \smc\ partially ordered vector space, then we let $\posextmeasts$ denote the collection of all positive $\osext$-valued measures $\npm:\borel\to\pososext$.

We distinguish the following regularity properties.

\begin{definition}\label{2_def:regularity_of_measures}
Let $\ts$ be a locally compact Hausdorff space, let $\os$ be a \mc\ partially ordered vector space, and let $\npm\in\posextmeasts$. Then $\npm$ is called:
\begin{enumerate}
\item\label{2_part:regularity_of_measures_1}
a \emph{Borel measure \uppars{on $\ts$}} when $\npm(K)\in \os$ for all compact subsets $K$ of~$\ts$;
\item\label{2_part:regularity_of_measures_2}
\emph{inner regular at $\mss\in\borel$} when $\npm(\mss)=\bigvee\{\npm(K): K\ \text{is compact and}\ K\subseteq \mss\}$ in~$\osext$;
\item\label{2_part:regularity_of_measures_3}
\emph{\wir\ at $\mss\in\borel$} when $\npm(\mss)=\bigvee\{\npm(\mss\cap K) :  K \text{ is compact}\}$ in~$\osext$;
\item\label{2_part:regularity_of_measures_4}
\emph{outer regular at $\mss\in\borel$} if $\npm(\mss)=\bigwedge\{\npm(V): V\ \text{is open and}\ \mss\subseteq V\}$ in~$\osext$;
\item\label{2_part:regularity_of_measures_5}
a \emph{regular Borel measure \uppars{on $\ts$}} when it is a Borel measure on $\ts$ that is inner regular at all open subsets of $\ts$ and outer regular at all Borel sets;
\item\label{2_part:regularity_of_measures_6}
a \emph{quasi-regular Borel measure \uppars{on $\ts$}} when it is a Borel measure on $\ts$ that is inner regular at all open subsets of $\ts$ and \wir\ at all Borel sets.
\end{enumerate}
\end{definition}

Note that the three subsets of $\os$ occurring in the above definitions are all directed in the appropriate directions, so that the two suprema and the infimum exist in $\osext$ as a consequence of the \mc ness of $\os$; \smc ness would not have sufficed here. In the sequel, there will be similar cases where the fact that a set is directed guarantees that a supremum or infimum exists, but we refrain from observing this explicitly each and every time.

To be entirely consistent with the terminology in \cref{2_sec:preliminaries}, we would have to speak of, for example, a quasi-regular $\pososext$-valued Borel measure, but we shall consider the codomains of the measures in the remainder of this paper to be tacitly understood from the context.

We let $\rcompacto$ denote the collection of relatively compact open subsets of $\ts$, i.e., the collection of open subsets of $\ts$ with compact closure. Then $\rcompacto$ is closed under finite intersections and finite unions.
Since every relatively compact open subset of $\ts$ is contained in a compact one, and every compact subset of $\ts$ is contained in a relatively compact open set, it is easy to see that $\npm$ is \wir\ at a Borel set $\mss$ if and only if  $\npm(\mss)=\bigvee\{\npm(\mss\cap V) : V\in\rcompacto\}$ in $\osext$.

The next small result relates the two regularity properties of a quasi-regular Borel measure. Its easy proof is omitted.

\begin{lemma}
	Let $\ts$ be a locally compact Hausdorff space, let $\os$ be a \mc\ partially ordered vector space, let $\npm\in\posextmeasts$, and let $\mss\in\borel$.
	If $\npm$ is inner regular at $\mss$, then $\npm$ is \wir\ at $\mss$.
\end{lemma}

\begin{remark}
	The terminology `quasi-regular' is admittedly not very descriptive, but it is hard to find a more appealing nomenclature. It is also used in \cite[p.~108]{wright:1969} as well as in \cite[p.~195]{wright:1971a} by the same author. It denotes different notions in these two papers, and these are also both different from ours.
\end{remark}

We shall write $\posextBmeas$ for the Borel measures on $\ts$, $\posextrBmeas$ for the regular Borel measures on  $\ts$, $\posextqrBmeas$ for the quasi-regular Borel measures on $\ts$, $\posrBmeas$ for the finite regular (Borel) measures on $\ts$, and $\posqrBmeas$ for the finite quasi-regular (Borel) measures on $\ts$. Every finite measure is, of course, a Borel measure.

If $U$ is an open subset of $\ts$, then $U$ is a locally compact Hausdorff space in the induced topology. We shall denote its Borel $\sigma$-algebra by $\borel_U$. It is then easy to see that $\borel_U=\{\mss\in\borel : \mss\subseteq U\}$, so that we can identify $\borel_U$ with a subset of $\borel$.
 If $\npm\in\posextmeasts$, then we let $\npm_U$ denote the restriction of $\npm$ to the subset $\borel_U$ of $\borel$. It is then a positive $\osext$-valued measure on $\borel_U$.

We collect a few technical observations in the following result. The proofs are elementary.

\begin{lemma}\label{2_res:technical_results_about_restrictions_of_measures}
Let $\ts$ be a locally compact Hausdorff space, let $\os$ be a \mc\ partially ordered vector space, let $\npm\in\posextmeasts$, and let $U$ be an open subset of $\ts$.
\begin{enumerate}
\item\label{2_part:technical_results_about_restrictions_of_measures_1}
If $\npm$ is a Borel measure on $\ts$, then $\npm_U$ is a Borel measure on $U$.
\item\label{2_part:technical_results_about_restrictions_of_measures_2}
If $\npm$ is inner regular at $\mss\in\borel_U$, then so is $\npm_U$; and conversely.
\item\label{2_part:technical_results_about_restrictions_of_measures_3}
If $\npm$ is outer regular at $\mss\in\borel_U$, then so is $\npm_U$; and conversely.
\item\label{2_part:technical_results_about_restrictions_of_measures_4}
If $\npm$ is a Borel measure on $\ts$ and if $U\in\rcompacto$, then $\npm_U$ is a finite Borel measure on $U$.
\item\label{2_part:technical_results_about_restrictions_of_measures_5}
If $\npm$ is a Borel measure on $U$ for all $U\in\rcompacto$, then $\npm$ is a finite Borel measure on $U$ for all $U\in\rcompacto$.
\end{enumerate}
\end{lemma}

The following is then clear.

\begin{lemma}\label{2_res:restricted_measure_is_regular_borel_measure}
Let $\ts$ be a locally compact Hausdorff space and let $\os$ be a \mc\ partially ordered vector space.

If $\npm\in\posextmeasts$ is a regular Borel measure on $\ts$, then, for all open subsets $U$ of $\ts$, $\npm_U$ is a regular Borel measure on $U$.
\end{lemma}

If $\os=\RR$, then every regular Borel measure on $\ts$ is inner regular at the $\sigma$-finite Borel sets; see \cite[Proposition~7.5]{folland_REAL_ANALYSIS_SECOND_EDITION:1999}. If $\os$ is normal, an analogous result holds in our case for Borel sets of finite measure. The step from the case of finite measure to the $\sigma$-finite case as in the proof of \cite[Proposition~7.5]{folland_REAL_ANALYSIS_SECOND_EDITION:1999} uses the linear ordering of $\RR$ and does not seem to have an obvious analogue in the general situation.

\begin{proposition}\label{2_res:mu_is_inner_regular_at_finite_borel_sets}
Let $\ts$ be a locally compact Hausdorff space, let $\os$ be a \mc\ and normal partially ordered vector space, and let $\npm\in\posextrBmeas$.

Then $\npm$ is inner regular at all Borel sets of finite measure.
\end{proposition}

\begin{proof}
Let $\mss \in\borel$ be such that $\npm(\mss )\in \os$. We are to prove that
\[
\npm(\mss )=\bigvee\{\npm(K): K\ \text{is compact and}\ K\subseteq \mss \}.
\]

It is clear that $\npm(\mss )$ is an upper bound of $\{\npm(K) : K\ \text{is compact and}\ K\subseteq \mss \}$. Using the normality of $\os$, we shall show that it is the least upper bound.

Choose and fix $x^\prime\in\pos{(\ocdual{E})}$ and $\varepsilon>0$.

Since $\npm(\mss )\in \os$, the outer regularity of $\npm$ at $\mss $ can be written as
\[
\npm(\mss )=\bigwedge\{\npm(V): V\text{ is open, }\mss \subseteq V\text{, and }\npm(V)\in\os\}.
\]
\cref{2_res:inf_and_sup_via_order_dual} then shows that
\[
\f{\npm(\mss ),x^\prime}=\inf\{\f{\npm(V),x^\prime}: V\text{ is open, }\mss \subseteq V\text{, and }\npm(V)\in\os\}.
\]
We can then choose and fix an open subset $V$ such that $\mss \subseteq V$, $\npm(V)\in\os$, and
\[
\f{\npm(V),x^\prime}<\f{\npm(\mss ),x^\prime}+\varepsilon/3.
\]
Using the outer regularity of $\npm$ at $V\setminus \mss $, which also has finite measure, we can likewise choose and fix an open subset $W$ such that $V\setminus \mss \subseteq W$, $\npm(W)\in E$, and
\[
\f{\npm(W),x^\prime}<\f{\npm(V\setminus \mss ),x^\prime}+\varepsilon/3.
\]
By replacing $W$ with the open subset $W\cap V$ we may and shall suppose that $W\subseteq V$. We have
\[
\f{\npm(W),x^\prime}<\f{\npm(V\setminus \mss ),x^\prime}+\varepsilon/3=\f{\npm(V),x^\prime}-\f{\npm(\mss ),x^\prime}+\varepsilon/3<2\varepsilon/3.
\]
Using the inner regularity of $\npm$ at the open subset $V$ and the normality of $\os$, an application of \cref{2_res:inf_and_sup_via_order_dual} enables us to choose and fix a compact subset $K_1$ such that $K_1\subseteq V$ and
\[
\f{\npm(V),x^\prime}<\f{\npm(K_1),x^\prime}+\varepsilon/3.
\]
Set $K_2\coloneqq K_1\setminus W$. Then $K_2$ is compact and
\begin{align*}
K_2&=K_1\cap W^\comp\\
&\subseteq K_1 \cap (V\setminus \mss )^\comp\\
&=(K_1\cap V^\comp)\cup(K_1\cap \mss )\\
&=K_1\cap \mss \\
&\subseteq \mss .
\end{align*}
We also have
\begin{align*}
V\setminus K_2&=V\cap(K_1^\comp\cup W)\\
&=(V\cap K_1^\comp)\cup (V\cap W)\\
&=(V\setminus K_1)\cup W.
\end{align*}

Combining all of the above, we see that
\begin{align*}
\f{\npm(\mss )-\npm(K_2),x^\prime}&=\f{\npm(\mss \setminus K_2),x^\prime}\\
&\leq\f{\npm(V\setminus K_2),x^\prime}\\
&=\f{\npm\left[(V\setminus K_1)\cup W\right],x^\prime}\\
&\leq\f{\npm(V\setminus K_1),x^\prime}+\f{\npm(W),x^\prime}\\
&=\f{\npm(V),x^\prime}-\f{\npm(K_1),x^\prime}+\f{\npm(W),x^\prime}\\
&<\varepsilon + 2\varepsilon/3\\
&=\varepsilon.
\end{align*}
All in all, we have found a compact subset $K_2$ of $\mss $ (depending on $x^\prime$ and $\varepsilon$) such that $\f{\npm(\mss ),x^\prime}< \f{\npm(K_2),x^\prime}+3\epsilon$. Hence we conclude, for our fixed $x^\prime\in\pos{(\ocdualos)}$, that
\[
\f{\npm(\mss ),x^\prime}=\sup\{\f{\npm(K),x^\prime} : K \text{ is compact and } K\subseteq \mss \}.
\]
Since this is true for all $x^\prime\in\pos{(\ocdualos)}$, an appeal to \cref{2_res:inf_and_sup_via_order_dual} completes the proof.
\end{proof}

We can use this to prove our next result.

\begin{corollary}\label{2_res:inner_regularity_of_quasi-regular_measure_for_normal_spaces}
Let $\ts$ be a locally compact Hausdorff space, let $\os$ be a \mc\ and normal partially ordered vector space, and let $\npm\in\posextmeasts$ be such that, for all $U\in\rcompacto$, the restricted measure $\npm_U$ is a regular Borel measure on $U$.

Suppose that $\mss$ is a Borel set such that $\npm$ is \wir\ at $\mss$. Then $\npm$ is inner regular at $\mss$.
\end{corollary}

\begin{proof}
If $U\in\rcompacto$, then $\mss\cap U\in\borel_U$ has finite measure. Hence an application of \cref{2_res:mu_is_inner_regular_at_finite_borel_sets} to the regular Borel measure $\npm_U$ on $U$ shows that
\[
\npm(\mss \cap U)=\bigvee\{\npm(K) :  K \text{ is compact and } K\subseteq \mss \cap U\}.
\]
Hence
\begin{align*}
\npm(\mss )&=\bigvee\{\npm(\mss \cap U): U\in\rcompacto\}\\
&=\bigvee\{\npm(K) : K \text{ is compact and }  K\subseteq \mss \cap U \text{ for some } U\in\rcompacto\}\\
&=\bigvee\{\npm(K) : K \text{ is compact and }\ K\subseteq \mss\}
\end{align*}
in $\osext$.
The last equation holds since the local compactness of $\ts$ implies that, for any compact subset $K$ with $K\subseteq \mss$, there exists a relatively compact open subset $U$ such that $K\subseteq U$, and then $K\subseteq \mss \cap U$.
\end{proof}

This has the following consequence.

\begin{corollary}\label{2_res:mu_is_inner_regular_at_finite_borel_sets_if_all_mu_u_are_regular_borel_measures}
Let $\ts$ be a locally compact Hausdorff space, let $\os$ be a \mc\ and normal partially ordered vector space, and let $\npm\in\posextmeasts$ be such that, for all $U\in\rcompacto$, the restricted measure $\npm_U$ is a regular Borel measure on $U$.
\begin{enumerate}
\item\label{2_part:mu_is_inner_regular_at_finite_borel_sets_if_all_mu_u_are_regular_borel_measures_1}
If $\npm$ is \wir\ at all Borel sets, then $\npm$ is inner regular at all Borel sets.
\item\label{2_part:mu_is_inner_regular_at_finite_borel_sets_if_all_mu_u_are_regular_borel_measures_2}
If $\npm$ is a finite measure that is \wir\ at all Borel sets, then $\npm$ is both inner regular and outer regular at all Borel sets.
\end{enumerate}
\end{corollary}

\begin{proof}
Part~\ref{2_part:mu_is_inner_regular_at_finite_borel_sets_if_all_mu_u_are_regular_borel_measures_1} is immediate from \cref{2_res:inner_regularity_of_quasi-regular_measure_for_normal_spaces}.
In part~\ref{2_part:mu_is_inner_regular_at_finite_borel_sets_if_all_mu_u_are_regular_borel_measures_2}, the inner regularity of $\mu$ at all Borel sets is clear from part~\ref{2_part:mu_is_inner_regular_at_finite_borel_sets_if_all_mu_u_are_regular_borel_measures_1}. For the outer regularity, let $\mss\in\borel$. Then the inner regularity of $\npm$ at $\mss^\comp$ yields that
\begin{align*}
\npm(\mss)&=\npm(\ts)-\npm(\mss^\comp)\\
&=\npm(\ts)-\bigvee\{\npm(K): K \text{ is compact and }K\subseteq \mss^\comp\}\\
&=\npm(\ts)-\bigvee\{\npm(\ts)-\npm(K^\comp): K \text{ is compact and }K\subseteq \mss^\comp\}\\
&=\bigwedge\{\npm(K^\comp): K \text{ is compact and }\mss\subseteq K^\comp\}\\
&\geq\bigwedge\{\npm(U): U \text{ is open and }\mss\subseteq U\}\\
&\geq\npm(\mss).
\end{align*}
Hence $\npm$ is outer regular at $\mss$.
\end{proof}

\section{Riesz representation theorem for $\contcts$: normed case}\label{2_sec:riesz_representation_theorems_for_contcts_normed_case}

\noindent In this section, we establish our first analogue of the classical Riesz  rep\-re\-sen\-ta\-tion  theorem for $\contcts$. As in the classical result, the representing measure in it can be infinite. It applies when the codomain of the positive operator is a suitable normed space; every Banach lattice with an order continuous norm will do.  The proof relies on the material on vector-valued outer measures in \cite[Section~5]{de_jeu_jiang:2021a}, to which we refer for details.

We shall employ the following customary notation.

\begin{definition}\label{2_def:precedes_notation}
If $\ts$ is a locally compact Hausdorff space and $S$ is an subset of $\ts$, then we shall write $f\prec S$ to denote that $f\in\contcts$, that $0\leq f(\pt)\leq 1$ for all $\pt\in\ts$, and that $\supp(f)\subseteq S$. We shall write $S\prec f$ to denote that $f\in\contcts$, that $0\leq f(\pt)\leq 1$ for all $\pt\in\ts$, and that $f(\pt)=1$ for all $\pt\in S$.
\end{definition}

\begin{theorem}[Riesz representation theorem for $\contcts$: normed case]\label{2_res:riesz_representation_theorem_for_contcts_normed_case}
Let $\os$ be a \mc\ partially ordered vector space that is also a normed space. Suppose that $\os$ has the following properties:
\begin{enumerate}
	\item\label{2_req:riesz_representation_theorem_for_contcts_normed_case_1}
	if $x_\lambda\downarrow 0$ in $\os$, then $\norm{x_\lambda}\to 0$;
	\item\label{2_req:riesz_representation_theorem_for_contcts_normed_case_2}
	if $\seq{x}\subset\posos$ is such that $\sum_{n=1}^\infty\norm{x_n}<\infty$, then  $\bigvee_{N=1}^\infty\sum_{n=1}^N x_n$, which exists in $\osext$, is actually finite.
	\item\label{2_req:riesz_representation_theorem_for_contcts_normed_case_3}
	if $x\in\os$ is such that $\f{x,x^\prime}\geq 0$ for all norm bounded positive $\sigma$-order continuous functionals on $\os$, then $x\in\posos$.
\end{enumerate} 	
Let $\ts$ be a non-empty locally compact Hausdorff space, and let $\posmap:\contcts\rightarrow\os$ be a positive operator.

Then there exists a unique regular Borel measure $\npm:\borel\to\pososext$ on the Borel $\sigma$-algebra $\borel$ of $\ts$ such that
\[
\posmap(f)=\ointm{f}
\]
for all $f\in\contcts$. If $V$ is a non-empty open subset of $\ts$, then
\begin{equation}\label{2_eq:measure_of_open_subsets_infinite_normed_case}
\npm(V)=\bigvee\{\posmap(f) : f\prec V\}
\end{equation}
in $\osext$. If $K$ is a compact subset of $\ts$, then
\begin{equation}\label{2_eq:measure_of_compact_subsets_infinite_normed_case}
\npm(K)=\bigwedge\{\posmap(f) : K\prec f\}
\end{equation}
in $\osext$. Hence $\npm$ is a finite measure if and only if
$\{\posmap(f) : f\in\pos{\contcts},\,\norm{f}\leq 1\}$ is bounded above in $\os$; this is automatically the case if $\ts$ is compact.

If $\npm$ is a finite measure and if the norm of $\os$ is monotone on $\posos$, then $\posmap$ is continuous with respect to the supremum norm topology on $\contcts$ and the norm topology on $\os$. In this case, $\norm{\posmap}\leq 2\norm{\npm(\ts)}$. If $\os$ is a normed vector lattice, then it is a Banach lattice with an order continuous norm, and $\norm{\posmap}=\norm{\npm(\ts)}$.
\end{theorem}

Let us mention explicitly that under~\ref{2_req:riesz_representation_theorem_for_contcts_normed_case_1} it is not required that $\norm{x_\lambda}\downarrow 0$.

In the context of a normed vector lattice $\os$, the requirement under~\ref{2_req:riesz_representation_theorem_for_contcts_normed_case_2}
is the so-called Riesz-Fischer property. It characterises the Banach lattices among the normed vector lattices; see \cite[Theorem~16.2]{zaanen_INTRODUCTION_TO_OPERATOR_THEORY_IN_RIESZ_SPACES:1997}, for example.

Before we embark on the lengthy proof, let us note that not only the real numbers satisfy all requirements in the theorem, but that this is, in fact, the case for every Banach lattice with an order continuous norm. Indeed, the first requirement is then satisfied by definition, and the second one is a consequence of the fact that the sum of a convergent series with positive terms in a normed vector lattice is also the supremum of the sequence of its partial sums; see, e.g., \cite[Theorem~15.3]{zaanen_INTRODUCTION_TO_OPERATOR_THEORY_IN_RIESZ_SPACES:1997}. The well-known validity of the third re\-quire\-ment and the \Dc ness were already observed in \cite[Proposition~3.10]{de_jeu_jiang:2021a}.

The global layout of the following proof of \cref{2_res:riesz_representation_theorem_for_contcts_normed_case} is modelled after that of \cite[Theorem~38.3]{aliprantis_burkinshaw_PRINCIPLES_OF_REAL_ANALYSIS_THIRD_EDITION:1998} for the real case;
the part where it is established that the measure as constructed actually represents the positive operator is essentially that of \cite[Proposition~7.2.11]{cohn_MEASURE_THEORY_BIRKHAUSER_REPRINT:1993} for the real case. Although our statement is a little stronger, the proof of \cref{2_eq:measure_of_compact_subsets_infinite_normed_case} is essentially that of the corresponding statement in the real case in \cite[Theorem~7.2]{folland_REAL_ANALYSIS_SECOND_EDITION:1999}.

The arguments in the proof rather resemble those for the real case, but there are certainly modifications. For example, the proof that the set map $\pom$ in the proof below is an outer measure requires a special argument in the vector-valued case. Furthermore, we shall use the fact that the space is Archimedean a few times. In the real case it is not always so obvious in the various proofs in the literature that this is a property of the real numbers that is crucial for the Riesz representation theorem. For example, the part of the proof of \cite[Theorem~7.2]{folland_REAL_ANALYSIS_SECOND_EDITION:1999} where it is shown that the measure represents the integral, concludes with an argument invoking the lattice structure of $\RR$. That does not apply in the more general context. To circumvent this the Archimedean property is used.

\begin{proof}[Proof of \cref{2_res:riesz_representation_theorem_for_contcts_normed_case}]
We define a set map $\pom:2^{\ts}\rightarrow\osext$ as follows. Set $\npm(\emptyset)\coloneqq0$, and set
\begin{equation}\label{2_eq:define_the_measure_of_open sets}
\npm(V)\coloneqq\bigvee\{\posmap(f) : f\prec V\}
\end{equation}
in $\osext$ for all non-empty open subsets of $\ts$. Note that the set in the right hand side of this equation is upward directed as a consequence of the positivity of $\posmap$, so that its supremum exists in $\osext$. Indeed, if $f_1,f_2\prec V$, then $f_1\vee f_2\prec V$, and $\posmap(f_1\vee f_2)\geq \posmap(f_1) $ and $\posmap(f_1\vee f_2)\geq\posmap(f_2)$.

Evidently, $\npm(V)\geq 0$ for all open subsets $V$ of $\ts$, and $\npm(V_1)\leq\npm(V_2)$ if $V_1$ and $V_2$ are open subsets of $\ts$ such that $V_1\subseteq V_2$.

For an arbitrary subset $\mss$ of $\ts$, we set
\begin{equation}\label{2_eq:define_an_outer_measure_on_X}
\pom(\mss)\coloneqq\bigwedge\{\npm(V) : V \text{ is open and }\mss\subseteq V\}.
\end{equation}
in $\osext$. Note that the set in the right hand side of this equation is downward directed, so that the infimum indeed exists in $\osext$. Indeed, if $V_1,V_2$ are open subsets of $\ts$ such that $\mss\subseteq V_1,V_2$, then $\mss\subseteq V_1\cap V_2$, and $\pom(V_1\cap V_2)\leq\pom(V_1)$ and $\pom(V_1\cap V_2)\leq\pom(V_2)$ as a consequence of the monotonicity of $\npm$ on the open subsets of $\ts$ that we have already observed. This monotonicity on the open subsets of $\ts$ also shows that $\npm(V)=\pom(V)$ for all open subsets $V$ of $\ts$. When $V$ is an open subset of $\ts$, then we shall write $\npm(V)$ for this common element of $\pososext$.

Evidently, $\pom(\mss)\geq 0$ for all subsets $\mss$ of $\ts$, and $\pom(\mss_1)\leq\pom(\mss_2)$ if $\mss_1$ and $\mss_2$ are subsets of $\ts$ such that $\mss_1\subseteq \mss_2$.

\emph{The first main step in the proof consists of showing that $\pom$ is an $\pososext$-valued outer measure in the sense of \textup{\cite[Definition~5.1]{de_jeu_jiang:2021a}}.}

We shall now proceed to show this.

Since $\pom(\emptyset)=0$ and since the monotonicity of $\pom$ has already been established, we are left with the $\sigma$-sub-additivity.

We first establish the $\sigma$-sub-additivity of $\pom$ for open subsets of $\ts$, on which it coincides with $\npm$. Let $\seq{V}$ be a sequence of open subsets of $\ts$. Then  $\bigcup_{n=1}^\infty V_n$ is open, and  we are to show that $\npm\left(\bigcup_{n=1}^\infty V_n\right)\leq\bigvee_{N=1}^\infty \sum_{n=1}^N \npm(V_n)$ in $\osext$. In doing so, we may and shall suppose that all $V_n$ are non-empty.

Suppose that $f\prec \bigcup_{n=1}^\infty V_n$. Since $\supp(f)$ is compact, there exists $m\geq 1$ such that $\supp(f)\subseteq\bigcup_{n=1}^m V_n$. Then \cite[Theorem~2.13]{rudin_REAL_AND_COMPLEX_ANALYSIS_THIRD_EDITION:1987} furnishes $f_1,\dotsc,f_m$ in $\contcts$ such that $f_n\prec V_n$ for $n=1,\dotsc,m$ and $\sum_{n=1}^m f_n(\pt)=1$ for all $\pt\in\supp(f)$. Then clearly $f\leq \sum_{n=1}^m f_n$, so
\[
\posmap(f)\leq\sum_{n=1}^m\posmap(f_n)\leq \sum_{n=1}^m \npm(V_n)\leq\bigvee_{N=1}^\infty\sum_{n=1}^N \npm(V_n).
\]
in $\osext$.
This is true for all $f\prec \bigcup_{n=1}^\infty V_n$, so we have
\[
\npm\left(\bigcup_{n=1}^\infty V_n\right)\leq\bigvee_{N=1}^\infty\sum_{n=1}^N \npm(V_n)
\]
in $\osext$. Hence $\pom$ is $\sigma$-sub-additive on the open subsets of $\ts$.

We shall now establish the $\sigma$-sub-additivity of $\pom$ on arbitrary subsets of $\ts$.

Let $\seq{\mss}$ be sequence of subsets of $\ts$. We show that $\pom(\bigcup_{n=1}^\infty \mss_n)\leq\bigvee_{N=1}^\infty \sum_{n=1}^N \pom(\mss_n)$ in $\osext$. In doing so, we may and shall suppose that
\[
x\coloneqq\bigvee_{N=1}^\infty \sum_{n=1}^N \pom(\mss_n)
\]
is actually an element of $\os$. Then $\pom(\mss_n)$ is also finite for $n\geq 1$, so that, by part~(3) of \cite[Lemma~2.3]{de_jeu_jiang:2021a},
\[
 \pom(\mss_n)=\bigwedge\{\npm(V): V \text{ is open, } \mss_n\subseteq V\text{, and }\npm(V)\in\os\}.
\]
That is, $(\npm(V)-\pom(\mss_n))\downarrow 0$ as $V$ runs over the open neighbourhoods of $\mss_n$. By the assumption under~\ref{2_req:riesz_representation_theorem_for_contcts_normed_case_1} in the theorem, we also have $\norm{\npm(V)-\pom(\mss_n)}\to 0$.

After these preparations, we fix a norm bounded positive $\sigma$-order continuous functional $x^\prime$ on $\os$ and $\varepsilon>0$.

For each $n\geq 1$, we can then find an open neigbourhood $V_n$ of $\mss_n$ such that $\npm(V_n)$ is finite and
$\norm{\npm(V_n)-\pom(\mss_n)}<\varepsilon/2^n$. As a consequence of the assumption under~\ref{2_req:riesz_representation_theorem_for_contcts_normed_case_2} in the theorem, we can set
\[
y\coloneqq \bigvee_{N=1}^\infty\sum_{n=1}^N\left(\npm(V_n)-\pom(\mss_n)\right)
\]
as an element of $\os$.

Let $N\geq 1$. Then
\begin{align*}
\left[\sum_{n=1}^N \npm(V_n)\right]-x
&\leq\sum_{n=1}^N\npm(V_n)-\sum_{n=1}^N\pom(\mss_n)\\
&=\sum_{n=1}^N\left(\npm(V_n)-\pom(\mss_n)\right)\\&\leq y,
\end{align*}
so that
\[
\sum_{n=1}^N \npm(V_n)\leq x+y.
\]
We conclude that $\bigvee_{N=1}^\infty\sum_{n=1}^N \npm(V_n)$ is finite, and that
\[
\bigvee_{N=1}^\infty\sum_{n=1}^N \npm(V_n)\leq x+y.
\]
Since $\bigcup_{n=1}^\infty \mss_n\subseteq\bigcup_{n=1}^\infty V_n$, we then see, using the monotonicity of $\pom$ and the $\sigma$-sub-additivity of $\npm$ on the open subsets of $\ts$, that
\[
\pom\left(\bigcup_{n=1}^\infty \mss_n\right)\leq\pom\left(\bigcup_{n=1}^\infty V_n\right)=\npm\left(\bigcup_{n=1}^\infty V_n\right)\leq\bigvee_{N=1}^\infty\sum_{n=1}^N \npm(V_n)\leq x+y.
\]
In particular, $\pom\left(\bigcup_{n=1}^\infty \mss_n\right)$ is finite.
Hence we see, using the $\sigma$-order continuity of $x^\prime$ in the second step, that
\begin{align*}
\lrf{\pom\left(\bigcup_{n=1}^\infty \mss_n\right), x^\prime}&\leq\f{x+y,x^\prime}\\
&=\f{x,x^\prime}+\sum_{n=1}^\infty\f{\npm(V_n)-\pom(\mss_n),x^\prime}\\
&\leq\f{x,x^\prime}+\sum_{n=1}^\infty\norm{\npm(V_n)-\pom(\mss_n)}\norm{x^\prime}\\
&\leq\f{x,x^\prime}+\varepsilon \norm{x^\prime}.
\end{align*}
Letting $\varepsilon\downarrow 0$, we conclude that
\[
\lrf{\pom\left(\bigcup_{n=1}^\infty \mss_n\right), x^\prime}\leq\f{x,x^\prime}.
\]
Since this is true for all norm bounded positive $\sigma$-order continuous functionals $x^\prime$, the assumption under~\ref{2_req:riesz_representation_theorem_for_contcts_normed_case_3} in the theorem implies that
\[
\pom\left(\bigcup_{n=1}^\infty \mss_n\right)\leq x=\bigvee_{N=1}^\infty\sum_{n=1}^N\npm(\mss_n).
\]
We have now shown that $\pom$ is an outer measure in the sense of \cite[Definition~5.1]{de_jeu_jiang:2021a}.

We recall from \cite[Definition~5.2]{de_jeu_jiang:2021a} that
a subset $\mss$ of $\pset$ is called $\pom$-measurable if $\pom(\msstwo)=\pom(\msstwo\cap\mss)+\pom(\msstwo\cap\mss^\comp)$ in $\osext$ for all $\msstwo\subseteq \pset$. According to \cite[Theorem~5.5]{de_jeu_jiang:2021a}, the $\pom$-measurable subsets of $\ts$ form a $\sigma$-algebra, and the restriction of $\pom$ to this $\sigma$-algebra is an $\pososext$-valued measure.

\emph{The next step is to show that every Borel set is $\pom$-measurable. The fact that the restriction of $\pom$ to the Borel $\sigma$-algebra is even a \emph{regular} Borel measure will simultaneously be established during the process.}

This step is itself broken up into a number of intermediate results.

\emph{We claim that $\pom(K)$ is finite for every compact subset $K$ of $\ts$.}

To see this, suppose that $K$ is a compact subset $\ts$. Using \cite[Theorem~2.7]{rudin_PRINCIPLES_OF_MATHEMATICAL_ANALYSIS_THIRD_EDITION:1976}, we can choose a relatively open subset $V$ of $\ts$ such that $K\subseteq V$. By Urysohn's Lemma (see \cite[Theorem~2.12]{rudin_PRINCIPLES_OF_MATHEMATICAL_ANALYSIS_THIRD_EDITION:1976}), there exists $g\in\contcts$ such that $\overline{V}\prec g$. From this, we see that
\[
\pom(K)\leq\pom(V)=\npm(V)=\bigvee\{\posmap(f) : f\prec V\}\leq\posmap(g),
\]
so that $\pom(K)$ is finite.

\emph{We claim that $\pom$ is finitely additive on the compact subsets of $\ts$.}	

To see this, it is sufficient to show that $\pom(K_1\cup K_2)\geq\pom(K_1)+\pom(K_2)$ for any two disjoint compact subsets $K_1$ and $K_2$ of $\ts$. Since $\ts$ is a Hausdorff space, there exist open subsets $U_1$ and $U_2$ of $\ts$ such that $K_1\subseteq U_1$, $K_2\subseteq U_2$, and $U_1\cap U_2=\emptyset$. Let $V$ be an arbitrary open neighbourhood of $K_1\cup K_2$. Then $K_1\subseteq V\cap U_1$, $K_2\subseteq V\cap  U_2$, and $(V\cap U_1)\cap(V\cap U_2)=\emptyset$. The latter disjointness implies that $f_1+f_2\prec V$ whenever $f_1\prec V\cap U_1$ and $f\prec V\cap U_2$. Using this, we see that
\begin{align*}
\pom(K_1)+\pom(K_2)&
\leq \npm(V\cap V_1)+\npm(V\cap V_2)\\
&=\sum_{i=1}^2\bigvee\{\posmap(f_i) :  f_i\prec V\cap V_i\}\\
&=\bigvee\{\posmap(f_1)+\posmap(f_2) : f_1\prec V\cap U_1 \text{ and }f_2\prec V\cap U_2\}\\
&=\bigvee\{\posmap(f_1+f_2) : f_1\prec V\cap U_1 \text{ and }f_2\prec V\cap U_2\}\\
&\leq\bigvee\{\posmap(f) : f\prec V\}\\
&=\npm(V).
\end{align*}
Recalling the defining \cref{2_eq:define_an_outer_measure_on_X}, we see that $\pom(K_1)+\npm(K_2)\leq \pom(K_1\cup K_2)$, as required.

\emph{We claim that $\pom(\supp(f))\geq\posmap(f)$ whenever $f\prec\ts$}.

Indeed, if $V$ is any open neighbourhood of $\supp(f)$, then $f\prec V$, so that $\posmap(f)\leq\npm(V)$. Recalling the defining \cref{2_eq:define_an_outer_measure_on_X}, we see that $\pom(\supp(f))\geq \posmap(f)$, as required.

\emph{We claim that
\begin{equation}\label{2_eq:inner_regularity_of_outer_measure_at_open_subsets}
\npm(V)=\bigvee\{\pom(K) : K \text{ is compact and } K\subseteq V\}
\end{equation}
in $\osext$ for every open subset $V$ of $\ts$}.

To see this, we may suppose that $V\neq\emptyset$. In that case, using the previous result and the monotonicity of $\pom$, we have
\begin{align*}
\npm(V)
&=\bigvee\{\posmap(f) :  f\prec V\}\\
&\leq \bigvee\{\pom(\supp{f}): f\prec V\}\\
&\leq\bigvee\{\pom(K): K \text{ is compact and } K\subseteq V\}\\
&\leq\pom(V)\\
&=\npm(V).
\end{align*}
This establishes our claim.

\emph{We claim that $\pom(K\cup V)=\pom(K)+\npm(V)$ whenever $K$ is a compact subset of $\ts$, $V$ is an open subset of $\ts$, and $K\cap V=\emptyset$.}

To see this, we use \cref{2_eq:inner_regularity_of_outer_measure_at_open_subsets} and the finite additivity of $\pom$ on the compact subsets of $\ts$ to justify that
\begin{align*}
\pom(K\cup V)&\leq \pom(K)+\npm(V)\\
&=\pom(K)+\bigvee\{\pom(K_V) : K_V \text{ is compact and } K_V\subseteq V\}\\
	&=\bigvee\{\pom(K)+\pom(K_V) : K_V \text{ is compact and }K_V\subseteq V\}\\
		&=\bigvee\{\pom(K\cup K_V) : K_V \text{  is compact and }K_V\subseteq V\}\\
	&\leq\pom(K\cup V).
\end{align*}
This establishes our claim.

\emph{We claim that every Borel set is $\pom$-measurable.}

To see this, it is sufficient to show that every open subset of $\ts$ is $\pom$-measurable. Let $V$ be an open subset of $\ts$. In order to show that $V$ is $\pom$-measurable, it is sufficient to show that $\pom(\msstwo)\geq\pom(\msstwo\cap V)+\pom(\msstwo\cap V^\comp)$ whenever $\msstwo$ is a subset of $\ts$ such that $\pom(\msstwo)$ is finite.

We first show this when $\msstwo$ is an open subset of $\ts$. In that case, let $K$ be a compact subset of $\msstwo\cap V$. Then $\msstwo=K\cup \left(\msstwo\setminus K\right)$ as a disjoint union. Since $\msstwo\setminus K$ is an open subset of $\ts$, our previous result shows that
\[
\npm(\msstwo)=\pom(K)+\npm(\msstwo\setminus K).
\]
Since $\msstwo\cap V^\comp\subseteq \msstwo\cap K^\comp$, this implies that
\[
\pom(K)+\pom(\msstwo\cap V^\comp)\leq\pom(K)+\npm(\msstwo\setminus K)=\npm(\msstwo).
\]
Using \cref{2_eq:inner_regularity_of_outer_measure_at_open_subsets} for the open subset $\msstwo\cap V$ of $\ts$, we see that
\[
\npm(\msstwo\cap V)+\pom(\msstwo\cap V^\comp)\leq\npm(\msstwo)
\]
for all open subsets $\msstwo$ of $\ts$, as desired. 

Suppose now that $\msstwo$ is an arbitrary subset of $\ts$ such that $\pom(\msstwo)$ is finite. If $U$ is an open neighbourhood of $\msstwo$ such that $\npm(U)$ is finite, then, using what we have just established in the first step, we see that
\[
\npm(U)\geq\npm(U\cap V )+\pom(U\cap V^\comp)\geq\pom(\msstwo\cap V)+\pom(\msstwo\cap V^\comp).
\]
Since this is evidently also true if $\npm(U)=\largest$, the defining \cref{2_eq:define_an_outer_measure_on_X} shows that
\[
\pom(\msstwo)\geq\pom(\msstwo\cap V)+\pom(\msstwo\cap V^\comp).
\]
This concludes the proof that all Borel sets are $\pom$-measurable.

\emph{In the remainder of the proof, we shall write $\npm$ for the restriction of $\pom$ to the Borel $\sigma$-algebra of $\ts$.}

We have already seen that $\npm$ is finite on the compact subsets of $\ts$. Furthermore, inner regularity at all open subsets of $\ts$ was established in \cref{2_eq:inner_regularity_of_outer_measure_at_open_subsets}, and outer regularity at all Borel sets was built in by the defining \cref{2_eq:define_an_outer_measure_on_X}.

\emph{Hence $\npm$ is a regular Borel measure on $\ts$.}

\emph{\Cref{2_eq:measure_of_open_subsets_infinite_normed_case} is simply the defining \cref{2_eq:define_the_measure_of_open sets}.}

\emph{We shall now establish \cref{2_eq:measure_of_compact_subsets_infinite_normed_case}.}

For this, we need a preparatory result that we shall also use in the sequel of the proof.

\emph{We claim that $\npm(K)\leq \posmap(f)$ whenever $K$ is a compact subset of $\ts$ and $f\in\contcts$ is such that $\indicator{K}\leq f$.}

To see this, fix $\varepsilon$ such that $0<\varepsilon<1$, and set $U_\varepsilon=\{\pt\in\pset : f(\pt)>1-\varepsilon\}$. Then $U_\varepsilon$ is open and $K\subseteq U_\varepsilon$. If $g\prec U_\varepsilon$, then $g\leq (1-\varepsilon)^{-1}f$, so that
\[
\npm(K)\leq \npm(U_\varepsilon)=\bigvee\{\posmap (g) : g\prec U_\varepsilon\}\leq (1-\varepsilon)^{-1}\posmap(f).
\]
Hence $\npm(K)\leq \posmap(f)+\varepsilon\npm(K)$ for all $\epsilon>0$. Since $\npm(K)$ is finite and $\os$ is Archimedean, this implies that $\npm(K)\leq \posmap(f)$.

\emph{Continuing with the proof of \cref{2_eq:measure_of_compact_subsets_infinite_normed_case}}, we see from what we have just demonstrated that $\npm(K)\leq\bigwedge\{\posmap(f) : K\prec f\}$. In order to establish the reverse inequality, let $V$ be an open neighbourhood of $K$. By Urysohn's Lemma (see \cite[Theorem~2.7]{rudin_REAL_AND_COMPLEX_ANALYSIS_THIRD_EDITION:1987}), there exists a function $f$ such that $K\prec f\prec V$. Then $\npm(V)\geq \posmap(f)$. \Cref{2_eq:define_an_outer_measure_on_X} then implies that
\[
\npm(K)=\bigwedge\{\npm(V) : V\text{  is open and }K\subseteq V\}\geq \bigwedge\{\posmap(f) : K\prec f\}.
\]

This concludes the proof of \cref{2_eq:measure_of_compact_subsets_infinite_normed_case}.

\emph{We shall now show that $\posmap(f)=\ointm{f}$ for all $f\in\contcts$.}

For this, we need a final preparatory result.

\emph{We claim that $\npm(K)\geq \posmap(f)$ whenever $K$ is a compact subset of $\ts$ and $f\in\contcts$ is such that $0\leq f\leq \indicator{K}$.}

To see this, let $V$ be an open neighbourhood of $K$. Then $f\prec V$, so that $\npm(V)\geq\posmap(f)$ by the defining \cref{2_eq:define_the_measure_of_open sets}, and then $\npm(K)\geq\posmap(f)$ by the defining  \cref{2_eq:define_an_outer_measure_on_X}.

\emph{Continuing with the proof that $\posmap(f)=\ointm{f}$ for all $f\in\contcts$}, we note that we may clearly suppose that $f\in\pos{\contcts}$. Fix $\varepsilon>0$. For all $n\geq 1$, define $f_n:\ts\to\RR$ by setting
\[
f_n(\pt)\coloneqq
\begin{cases}
0&\text{if }f(\pt)\leq (n-1)\varepsilon;\\
f(\pt)-(n-1)\varepsilon&\text{if }(n-1)\varepsilon<f(\pt)\leq n\varepsilon;\\
\varepsilon&\text{if }n\varepsilon<f(\pt).
\end{cases}
\]
A moment's thought shows that $f_n\in\pos{\contcts}$ and $f_n(\ts)\subseteq[0,\varepsilon]$ for all $n\geq 1$, and that $\sum_{n=1}^\infty f_n(\pt)=f(\pt)$, where the series is a finite sum for every $\pt\in\ts$. There exists $N\geq 1$ such that $f_n=0$ for all $n>N$, so that actually $f=\sum_{n=1}^N f_n$ is a finite sum. We set $K_0\coloneqq\supp(f)$ and $K_n\coloneqq\{\pt\in\ts : f(\pt)\geq n\varepsilon\}$ for $n\geq 1$.
Then $\varepsilon\indicator{K_n}\leq f_n\leq\varepsilon \indicator{K_{n-1}}$ for all $n\geq 1$. The first of these inequalities, written as $\indicator{K_n}\leq f_n/\varepsilon$, implies that $\npm(K_n)\leq \posmap(f_n)/\varepsilon$ for $n\geq 1$ by one of the claims established above. The second of these inequalities, written as $f_n/\varepsilon\leq\indicator{K_{n-1}}$, implies that $\npm(K_{n-1})\geq\posmap(f_n)/\varepsilon$ for all $n\geq 1$ by our most recent auxiliary result. Thus we have $\varepsilon\npm(K_n)\leq \posmap(f_n)\leq\varepsilon \npm(K_{n-1})$ for $n\geq 1$. It is evident from the monotonicity of the order integral that $\varepsilon\npm(K_n)\leq \ointm{f_n}\leq\varepsilon \npm(K_{n-1})$ for $n\geq 1$. Since $f=\sum_{n=1}^N f_n$, a summation of the inequalities shows that
\[
\varepsilon\sum_{i=1}^N\npm(K_n)\leq\posmap(f)\leq \varepsilon\sum_{i=1}^N\npm(K_{n-1})
\]
and that
\[
\varepsilon \sum_{i=1}^N\npm(K_n)\leq\ointm{f}\leq\varepsilon \sum_{i=1}^N\npm(K_{n-1}).
\]
This implies that $\posmap(f)-\ointm{f}\leq \varepsilon(\npm(K_0)-\npm(K_N))\leq\varepsilon\npm(K_0)=\varepsilon\npm(\supp(f))$.
Since $\varepsilon>0$ is arbitrary and $\os$ is Archimedean, we see that $\posmap(f)-\ointm{f}\leq 0$. The reverse inequality is likewise true by a similar argument, and we conclude that $\posmap(f)=\ointm{f}$ for all $f\in\contcts$, as required.

\emph{We turn to the uniqueness of $\npm$}.

Suppose that $\npm^\prime$ is another regular Borel measure on $\ts$ such that $\posmap(f)=\orderintegral{\ts}{f}{{\npm}^\prime}$ for all $f\in\contcts$. By regularity, it is sufficient to show that $\npm(K)=\npm^\prime(K)$ for all compact subsets $K$ of $\ts$.

Let $V$ be a open neighbourhood of $K$. The Urysohn Lemma furnishes  $f\in\contcts$ with $K\prec f\prec V$  (see \cite[Theorem~2.13]{rudin_REAL_AND_COMPLEX_ANALYSIS_THIRD_EDITION:1987}), and then
\[
\npm^\prime(K)=\orderintegral{\ts}{\indicator{K}}{{\npm}^\prime}\leq \orderintegral{\ts}{f}{{\npm}^\prime}=\posmap(f)=\ointm{f}
\leq\ointm{\indicator{V}}=\npm(V).
\]
Using the outer regularity of $\npm$ at $K$, we see that $\npm^\prime(K)\leq\npm(K)$. The reverse inequality likewise holds, and we conclude that $\npm$ is the unique regular Borel measure on $\ts$ that represents $\posmap$.

\emph{Clearly, $\npm$ is a finite measure if and only if $\npm(\ts)$ is finite. If the space $\ts$ is compact, then $\{\posmap(f) : f\in\pos{\contcts},\,\norm{f}\leq 1\}$ is obviously bounded above in $\os$, since it has $\posmap(\onefunction)$  as an upper bound.}

\emph{We prove the statements in the final paragraph of the theorem.}

Suppose that $\npm(\ts)$ is finite and that the norm of $\os$ is monotone on its positive cone. Take $f\in\contcts$ such that $\norm{f}\leq 1$. Then $\norm{\posmap (f)}\leq\norm{\posmap(\pos{f})}+\norm{\posmap(\negt{f})}\leq 2\norm{\posmap(\abs{f})}\leq 2 \norm{\npm(\ts)}$. Hence $\posmap$ is continuous with norm at most $2\norm{\npm(\ts)}$. When $\os$ is a vector lattice, then, as already observed preceding the proof, it is a Banach lattice as a consequence of the assumption under~\ref{2_req:riesz_representation_theorem_for_contcts_normed_case_2} in the theorem; its norm is order continuous by assumption~\ref{2_req:riesz_representation_theorem_for_contcts_normed_case_1}. Take $f\in\contcts$ such that $\norm{f}\leq 1$. Since $\posmap$ is positive, we have $\norm{\posmap(f)}=\norm{\abs{\posmap(f)}}\leq\norm{\posmap{\abs{f}}}\leq\norm{\npm(\ts)}$. Hence $\norm{\posmap}\leq\norm{\npm(\ts)}$. The order continuity of the norm on $\os$ implies that equality holds.
\end{proof}

Also for other $\os$ than the real numbers, it is easy to give examples of positive operators $\posmap:\contcts\to\os$ where \cref{2_res:riesz_representation_theorem_for_contcts_normed_case} applies and where the representing measure $\npm$ is infinite. Indeed, let $\ts$ be a locally compact Hausdorff space and let $\os$ be a Banach lattice with an order continuous norm. Suppose that $(\rho,e)$ is a pair where $\rho$ is a regular Borel measure $\rho:\borel\to\posRext$, and where $e\in\posos\setminus\{0\}$. Set $\posmap_{\rho,e}(f)\coloneqq\int_\ts f\di{\rho}\cdot e$ for $f\in\contcts$. Then $\posmap_{\rho,e}:\contcts\to\os$ is positive, and its representing measure $\npm_{\rho,e}$ is given by $\npm_{\rho,e}(\mss)=\rho(\mss)\cdot e$ for $\mss\in\borel$. Hence $\npm_{\rho,e}$ is infinite precisely when $\rho$ is infinite. Although this yields examples where the representing measures are infinite, these are hardly interesting. However, one can use such operators as building blocks for others. Indeed, if $(\rho_1,e_1),\dotsc,(\rho_n,e_n)$ are $n$ such pairs and $\rho_1$ is infinite, then $\bigvee_{i=1}^n \posmap_{\rho_i,e_i}\in\pos{\regularop{\contcts,\os}}$, and its representing measure is again infinite. The latter is immediate from the fact that $\bigvee_{i=1}^n \posmap_{\rho_i,e_i}\geq \posmap_{\rho_1,e_1}$ and \cref{2_eq:measure_of_open_subsets_infinite_normed_case}.
Subsequently, one can take positive linear combinations of such finite suprema to obtain a non-zero cone of positive operators from $\contcts$ into $\os$ such that every non-zero element of this cone has an infinite representing measure.

\section{Riesz representation theorems for $\contcts$: normal case}\label{2_sec:riesz_representation_theorems_for_contcts_normal_case}

\noindent In this section, we establish two representation theorems for positive operators  $\posmap:\contcts\to\os$ in the case where $\os$ is a \mc\ and normal partially ordered vector space. As can be seen from \cref{2_ex:combination_result_for_normality_and_monotone_completeness}, such spaces are quite common and diverse. Some of the given examples are algebras, but that plays no role for the representation theorems in this section. The positive operator $\posmap$ need merely be linear.

The idea is to use the classical Riesz representation theorem (a special case of \cref{2_res:riesz_representation_theorem_for_contcts_normed_case}) as a stepping stone by invoking the positive order continuous functionals on $\os$ and exploiting the normality of $\os$.

Of course, functionals can act only on finite elements of $\osext$, and this is the reason that we first prove a representation theorem when the set $\{\posmap(f) : f\prec \ts\}$ is supposed to be bounded above in $\os$. This will serve to guarantee that everything in sight is a finite element of $\osext$. It is thus that \cref{2_res:riesz_representation_theorem_for_contcts_finite_normal_case} is established, where the representing measure is finite (by assumption) and regular. The second representation theorem, valid for the general case, is then obtained from the finite case; see \cref{2_res:riesz_representation_theorem_for_contcts_normal_case}. Here the representing measure can be infinite, but regularity may then hold only locally. This should be compared to \cref{2_res:riesz_representation_theorem_for_contcts_normed_case}, where the also possibly infinite representing measure is always regular.

We shall have additional use for \cref{2_res:riesz_representation_theorem_for_contcts_finite_normal_case} for the finite case later on; see, e.g., the proofs of  \cref{2_res:riesz_representation_theorem_for_contots_normal_case,2_res:riesz_representation_theorem_for_contots_operators_on_hilbert_space}.

We start by introducing a notation.

Suppose that $\ms$ is a measurable space, that $\os$ is a \smc\ partially ordered vector space, and that $\npm:\alg\to\posos$ is a (finite-valued) set map. If $x^\prime:\os\to\RR$ is a positive functional, we define $\npmxp:\alg\to\posR$ by setting  $\npmxp(\mss)\coloneqq\lrf{\npm(\mss),x^\prime}$ for $\mss\in\alg$.

The following is immediate from the definitions and \cite[Proposition~3.9]{de_jeu_jiang:2021a}.
\begin{lemma}\label{2_res:vector_valued_measures_and_real_measures}
Let $\ms$ be a measurable space, let $\os$ be a \smc\ partially ordered vector space, and let $\npm:\alg\to\posos$ be a set map.
	\begin{enumerate}
		\item\label{2_part:vector_valued_measures_and_real_measures_1}
		If $\npm:\alg\to\posos$ is a measure, then $\npmxp:\alg\to\posR$ is a finite measure for all $x^\prime\in\pos{(\socdualos)}$.
		\item\label{2_part:vector_valued_measures_and_real_measures_2}
		If $\os$ is $\sigma$-normal, then the following are equivalent:
		\begin{enumerate}
			\item\label{2_part:vector_valued_measures_and_real_measures_2_1}
			$\npm:\alg\to\posos$ is a finite measure;
			\item\label{2_part:vector_valued_measures_and_real_measures_2_2}
			$\npmxp:\alg\to\posR$ is a finite measure for all $x^\prime\in\pos{(\socdualos)}$.
		\end{enumerate}	
	\end{enumerate}
\end{lemma}

Likewise, the following follows directly from the definitions and \cref{2_res:inf_and_sup_via_order_dual}.

\begin{lemma}\label{2_res:regular_vector_valued_measures_and_regular_real_measures}
	Let $\ts$ be a locally compact Hausdorff space with Borel $\sigma$-algebra $\borel$, let $\os$ be a \mc\ and normal partially ordered vector space, and let $\npm:\alg\to\posos$ be a set map.

Then the following are equivalent:
	\begin{enumerate}
		\item\label{2_part:regular_vector_valued_measures_and_regular_real_measures_1}
		$\npm:\alg\to\posos$ is a finite measure;
		\item\label{2_part:regular_vector_valued_measures_and_regular_real_measures_2}
		$\npmxp:\alg\to\posR$ is a finite measure for all $x^\prime\in\pos{(\ocdualos)}$.
	\end{enumerate}

If this is the case, then the following are equivalent for a Borel set $\mss$:
\begin{enumerate_alpha}
	\item\label{2_part:regular_vector_valued_measures_and_regular_real_measures_3}
	$\npm:\alg\to\posos$ is inner regular \uppars{resp.\ outer regular, resp.\ \wir} at $\mss$;
	\item\label{2_part:regular_vector_valued_measures_and_regular_real_measures_4}
	$\npmxp:\alg\to\posR$ is inner regular \uppars{resp.\ outer regular, resp.\ \wir} at $\mss$ for all $x^\prime\in\ocdualos$.
\end{enumerate_alpha}			
\end{lemma}

We shall have use for the following immediate consequence of the above two results. Of course, finite measures are always Borel measures, but we have still included the redundant adjective for reasons of uniformity.

\begin{corollary}\label{2_res:relation_between_vector-valued_and_real_measures}
	Let $\ts$ be a locally compact Hausdorff space with Borel $\sigma$-algebra $\borel$, let $\os$ be a \mc\ and normal partially ordered vector space, and let $\npm:\alg\to\posos$ be a set map.

Then the following are equivalent:
\begin{enumerate}
	\item\label{2_part:relation_between_vector-valued_and_real_measures_1}
	$\npm:\alg\to\posos$ is a finite regular \uppars{resp.\ quasi-regular} Borel measure on $\ts$;
	\item\label{2_part:relation_between_vector-valued_and_real_measures_2}
	$\npmxp:\alg\to\posR$ is a finite regular \uppars{resp.\ quasi-regular} Borel measure on $\ts$ for all $x^\prime\in\pos{(\ocdualos)}$.
\end{enumerate}
\end{corollary}

After these preparations, we can now establish our first representation theorem for \mc\ and normal partially ordered vector spaces. As may be clear from \cref{2_ex:combination_result_for_normality_and_monotone_completeness}, this is a rather large and varied class of spaces.

\begin{theorem}[Riesz representation theorem for $\contcts$: finite normal case]\label{2_res:riesz_representation_theorem_for_contcts_finite_normal_case}
	Let $\ts$ be a locally compact Hausdorff space, let $\os$ be a \mc\ and normal partially ordered vector space, and let $\posmap:\contcts\rightarrow\os$ be a positive operator such that $\{\posmap(f) : f\in\pos{\contcts},\,\norm{f}\leq 1\}$ is bounded above in $\os$; this is automatically the case if $\ts$ is compact.
	
	Then there exists a unique regular Borel measure $\npm:\borel\to\pososext$ on the Borel $\sigma$-algebra $\borel$ of $\ts$ such that
	\[
	\posmap(f)=\ointm{f}
	\]
	for all $f\in\contcts$. The measure $\npm$ is finite. If $V$ is a non-empty open subset of $\ts$, then
	\begin{equation}\label{2_eq:measure_of_open_subsets_finite_normal_case}
	\npm(V)=\bigvee\{\posmap(f) : f\prec V\}
	\end{equation}
	in $\os$. If $K$ is a compact subset of $\ts$, then
	\begin{equation}\label{2_eq:measure_of_compact_subsets_finite_normal_case}
	\npm(K)=\bigwedge\{\posmap(f) : K\prec f\}
	\end{equation}
	in $\os$.
\end{theorem}

\begin{remark}\label{2_rem:connection_with_sot_in_finite_normal_case}
	As noted in part~\ref{2_part:comparison_with_vector_measures_and_sot_sigma_additivity_2} of \cref{2_rem:comparison_with_vector_measures_and_sot_sigma_additivity}, when the space $\os$ in \cref{2_res:riesz_representation_theorem_for_contcts_finite_normal_case} consists of the regular operators on a Banach lattice with an order continuous norm, or of the self-adjoint operators in a strongly closed complex linear subspace of the bounded operators on a complex Hilbert space, then $\npm$ is strongly $\sigma$-additive.
\end{remark}

\begin{proof}[Proof of \cref{2_res:riesz_representation_theorem_for_contcts_finite_normal_case}]
If $\ts$ is compact, then $\{\posmap(f) : f\in\pos{\contcts},\,\norm{f}\leq 1\}$ is obviously bounded above in $\os$ by $\posmap(\onefunction)$.

For general locally compact $\ts$, suppose that $\{\posmap(f) : f\in\pos{\contcts},\,\norm{f}\leq 1\}$ is bounded above in $\os$. Motivated by the proof of \cref{2_res:riesz_representation_theorem_for_contcts_normed_case}, we  define the set map $\npm: \borel\to\posos$ by setting $\npm(\emptyset)\coloneqq0$,
\begin{align}
\npm(V)&\coloneqq\bigvee\{\posmap(f) : f\prec V\}\label{2_eq:out_of_the_blue_for_open_subsets}\\
\intertext{in $\osext$ for every non-empty open subset $V$ of $\ts$, and}
\npm(\mss)&\coloneqq\bigwedge\{\npm(V) : V \text{ is open and }\mss\subseteq V\}\label{2_eq:out_of_the_blue_for_borel_sets}\\
\end{align}
in $\osext$ for every Borel set $\mss$. The remarks that were made  in the beginning of the proof of \cref{2_res:riesz_representation_theorem_for_contcts_normed_case} apply here as well, showing that both $\npm(V)$ and $\npm(\mss)$ are well-defined, and that the two definitions agree on the open subsets of $\ts$. It is clear from the hypotheses that $\npm(\mss)$ is, in fact, a finite element of $\osext$ for every Borel set $\mss$, so that we can let functionals on $\os$ act on the range of $\npm$.

We shall show that the set map $\npm$ has all required properties.

To this end, we define, for every $x^\prime\in\pos{(\ocdualos)}$, the functional $\posmapxp : \contcts\to\RR$ by setting
\[
\posmapxp(f)\coloneqq\f{\posmap(f),x^\prime}
\]
for $f\in\contcts$. Then $\posmapxp$ is a positive functional on $\contcts$. The classical Riesz representation theorem (the case where $\os=\RR$ in \cref{2_res:riesz_representation_theorem_for_contcts_finite_normal_case}) informs us that, for all $x^\prime\in\pos{(\ocdualos)}$, there exists a unique regular Borel measure $\npnxp:\borel\to\posRext$ such that
\begin{equation}\label{2_eq:npnxp_represents_posmapxp}
\posmapxp(f)=\orderintegral{\ts}{f}{\!\npnxp}
\end{equation}
for all $f\in\contcts$. Since $\{\posmapxp(f): f\prec \ts\}$ is evidently bounded above in $\RR$, \cref{2_res:riesz_representation_theorem_for_contcts_normed_case} shows that $\npnxp$ is, in fact, a finite regular Borel measure on $\ts$ for all $x^\prime\in\pos{(\ocdualos)}$.

For each $x^\prime\in\pos{(\ocdualos)}$, we know from \cref{2_eq:measure_of_open_subsets_infinite_normed_case} that
\[
\npnxp(V)=\sup\{\posmapxp(f) : f\prec V\}
\]
for all non-empty open subsets $V$ of $\ts$. On the other hand, we see from the defining \cref{2_eq:out_of_the_blue_for_open_subsets} that
\[
\npmxp(V)=\sup\{\posmapxp(f) : f\prec V\}
\]
for all $x^\prime\in\pos{(\ocdualos)}$ and all non-empty open subsets $V$ of $\ts$. Hence $\npmxp(V)=\npnxp(V)$ for all non-empty open subsets $V$ of $\ts$ and all $x^\prime\in\pos{(\ocdualos)}$; this is trivially true for the empty subset.

Furthermore, the outer regularity of $\npnxp$ means that, for every Borel set $\mss$,
\[
\npnxp(\mss)=\inf\{\npnxp(V) : V\text{ is open and } \mss\subseteq V\}.
\]
Using the defining \cref{2_eq:out_of_the_blue_for_open_subsets} and what we have just observed for the open subsets of $\ts$,  we therefore see that, for every Borel set $\mss$,
\begin{align*}
\npmxp(\mss)&=\inf\{\npmxp(V) : V\text{ is open and } \mss\subseteq V\}\\
&= \inf\{\npnxp(V) : V\text{ is open and } \mss\subseteq V\}\\
&=\npnxp(\mss).
\end{align*}
We conclude that $\npmxp=\npnxp$ for all $x^\prime\in\pos{(\ocdualos)}$. Since we know that $\npnxp$ is a regular Borel measure on $\ts$ for all $x^\prime$, it now follows from \cref{2_res:relation_between_vector-valued_and_real_measures} that $\npm$ is a regular Borel measure on $\ts$.

It is clear from \cref{2_eq:out_of_the_blue_for_open_subsets} that $\npm$ is a finite measure.

\Cref{2_eq:measure_of_open_subsets_finite_normal_case} holds by construction.

If $K$ is a compact subset of $\ts$, then \cref{2_eq:measure_of_compact_subsets_infinite_normed_case} shows that
\[
\npmxp(K)=\npnxp(K)=\bigwedge\{\posmapxp(f): K\prec f\}.
\]
It then follows from \cref{2_res:inf_and_sup_via_order_dual} that
\[
\npm(K)=\bigwedge\{\posmap(f): K\prec f\},
\]
which is \cref{2_eq:measure_of_compact_subsets_finite_normal_case}.

Let $f\in\pos{\contcts}$, and choose a sequence $\seq{\varphi}$ in $\elemfunts$ such that $\varphi_n\uparrow f$ in $\os$ pointwise. Fix $x^\prime\in\pos{(\ocdualos)}$. Using the definition and \cref{2_eq:npnxp_represents_posmapxp}, where we already know that $\npmxp=\npnxp$, we see that
\[
\orderintegral{\ts}{\varphi_n}{\npmxp}\uparrow\orderintegral{\ts}{f}{\npmxp}=\posmapxp(f).
\]
On the other hand, it is clear from the definition of the integral of elementary functions that
\[
\orderintegral{\ts}{\varphi_n}{\npmxp}=\lrf{\ointm{\varphi_n},x^\prime}
\]
for all $n\geq 1$. Since $\ointm{\varphi_n}\uparrow\ointm{f}$ by definition, the $\sigma$-order continuity of $x^\prime$ therefore implies that
\[
\orderintegral{\ts}{\varphi_n}{\npmxp}=\lrf{\ointm{\varphi_n},x^\prime}\uparrow\lrf{\ointm{f},x^\prime}.
\]
We conclude that $\posmapxp(f)=\lrf{\ointm{f},x^\prime}$. Since this holds for all $x^\prime\in\pos{(\ocdualos)}$ and $\os$ is normal, it follows that $\posmap(f)=\ointm{f}$ for all $f\in\pos{\contcts}$. By linearity, this is then also true for general $f\in\contcts$.

The uniqueness of $\npm$ as an a priori possibly infinite regular Borel measure   rep\-re\-sen\-ting $\posmap$ follows as in the conclusion of the proof of \cref{2_res:riesz_representation_theorem_for_contcts_normed_case}.
\end{proof}

We can now also obtain a representation theorem when $\{\posmap(f):f\prec \ts\}$ need not be bounded above in $\os$. The idea is to start by using  \cref{2_res:riesz_representation_theorem_for_contcts_finite_normal_case} locally. If $U\in\rcompacto$ (that is: if $U$ is an open relatively compact subset of $\ts$), then there exists $\widetilde f\in\contcts$ such that $\overline{U}\prec \widetilde f$. This implies that $f\leq\widetilde f$ whenever $f\prec U$, so that $\posmap(f)\leq\posmap(\widetilde f)$ whenever $f\prec U$. We are now in the finite case that is covered by \cref{2_res:riesz_representation_theorem_for_contcts_finite_normal_case}, so that there exists a unique regular Borel measure $\npm_U$ on $U$ such that $\posmap(f)=\orderintegral{U}{f}{\npm_U}$ for all $f\in\contcts$ that are supported in $U$. These local measures $\npm_U$ on $U$ can then seen to be the restrictions of a global measure on $\ts$ using a technique that is also employed in \cite[proof of Theorem~1]{wright:1971a}.

The details are contained in the proof of the following theorem. As \cref{2_res:riesz_representation_theorem_for_contcts_finite_normal_case}, it applies to the variety of spaces in \cref{2_ex:combination_result_for_normality_and_monotone_completeness}.

\begin{theorem}[Riesz representation theorem for $\contcts$: normal case]\label{2_res:riesz_representation_theorem_for_contcts_normal_case}
	Let $\ts$ be a locally compact Hausdorff space, let $\os$ be a \mc\ and normal partially ordered vector space, and let $\posmap:\contcts\to\os$ be a positive operator.
	
	Then there exists a unique measure $\npm:\borel\to\pososext$ such that:
	\begin{enumerate}
		\item\label{2_part:riesz_representation_theorem_for_contcts_normal_case_1}
		For every $U\in\rcompacto$, the restriction $\npm_U$ of $\npm$ to $\borel_U$ is a regular Borel measure on $U$;
		\item\label{2_part:riesz_representation_theorem_for_contcts_normal_case_2}
		$\npm$ is \wir\ at all Borel sets;
		\item\label{2_part:riesz_representation_theorem_for_contcts_normal_case_3}
		$\posmap(f)=\ointm{f}$ for all $f\in\contcts$.
	\end{enumerate}
	This measure $\npm$ is a quasi-regular Borel measure.
	
	If $\widetilde\npm$ is a regular Borel measure on $\ts$ such that $\posmap(f)=\ointm{f}$ for all $f\in\contcts$, then $\widetilde\npm\geq\npm$. Furthermore, $\widetilde\npm$ and $\npm$ agree at all open subsets of $\ts$; at all compact subsets of $\ts$; at all Borel sets with finite $\widetilde\npm$-measure; and at all Borel sets with infinite $\npm$-measure.  	
\end{theorem}

\begin{remark}\label{2_rem:connection_with_sot_in_normal_case}\quad
	\begin{enumerate}
		\item\label{2_part:connection_with_sot_in_normal_case_1}
		For a partially ordered vector space of operators, it is not at all uncommon\textemdash see \cref{2_ex:combination_result_for_normality_and_monotone_completeness} and also \cite[Section~3]{de_jeu_jiang:2021a}\textemdash to be \mc\ and normal, so that \cref{2_res:riesz_representation_theorem_for_contcts_normal_case} applies.
		\item\label{2_part:connection_with_sot_in_normal_case_2}
		As noted in part~\ref{2_part:comparison_with_vector_measures_and_sot_sigma_additivity_2} of \cref{2_rem:comparison_with_vector_measures_and_sot_sigma_additivity}, when the space $\os$ in \cref{2_res:riesz_representation_theorem_for_contcts_normal_case} consists of the regular operators on a Banach lattice with an order continuous norm, or of the self-adjoint operators in a strongly closed complex linear subspace of the bounded operators on a complex Hilbert space, and when the measure $\npm$ is finite, then $\npm$ is strongly $\sigma$-additive.
		\item\label{2_part:connection_with_sot_in_normal_case_3}
		We mention the following consequence of the last part of \cref{2_res:riesz_representation_theorem_for_contcts_normal_case}: if $\mss$ is a Borel set such that $\widetilde\npm(\mss)\neq\npm(\mss)$, then $\mss$ is neither open nor compact, $\widetilde\npm(\mss)=\largest$, and $\npm(\mss)$ is finite. It also follows that, if $\posmap$ is represented by a finite regular Borel measure $\widetilde\npm$, then it must be the case that $\widetilde\npm=\npm$.
	\end{enumerate}
\end{remark}

\begin{proof}[Proof of \cref{2_res:riesz_representation_theorem_for_contcts_normal_case}]
	The uniqueness is easily taken care of. For this, we take a non-empty open subset $U$ of $\ts$, and we view $\contc{U}$ as a linear subspace of $\contcts$ by extending an $f$ in $\contc{U}$ to an element $f^\prime$ of $\contcts$ that is zero outside $U$. We obtain a positive operator $\posmap_U:\contc{U}\to\os$ by setting $\posmap_U(f)\coloneqq \posmap(f^\prime)$ for $f\in\contc{U}$, and then $\posmap_U(f)=\posmap(f^\prime)=\ointm{f^\prime}=\orderintegral{U}{f}{\npm_U}$ for $f\in\contc{U}$. Because $\npm_U$ is supposed to be a regular Borel measure on $U$, the uniqueness statement in \cref{2_res:riesz_representation_theorem_for_contcts_finite_normal_case} implies that $\npm_U$ is uniquely determined. Since $\npm$ is \wir\ at all Borel sets, i.e., since
	\[
	\npm(\mss)=\bigvee\{\npm(\mss\cap U) : U\in\rcompacto\}=\bigvee\{\npm_U(\mss\cap U) : U\in\rcompacto\},
	\]
	for every Borel set $\mss$, we see that $\npm$ itself is unique.
	
	This proof of the uniqueness also shows how to find $\mu$, as follows. If $U\in\rcompacto$, then, as explained preceding the theorem, there exists a unique regular Borel measure $\npm_U$ on $U$ such that $\posmap(f)=\orderintegral{U}{f}{\npm_U}$ for all $f\in\contc{U}$.
	
	We shall now patch these $\npm_U$ together to obtain a measure on $\ts$.
	
	As a preparation for this, we note the following.
	
    If $U,V\in\rcompacto$ and $U\cap V\neq\emptyset$, then the restriction of $\npm_U$ to $U\cap V$ is a regular Borel measure on $U\cap V$ that represents the restriction of $\posmap$ to $\contc{U\cap V}$. Since the same holds for $\npm_V$, the uniqueness statement of \cref{2_res:riesz_representation_theorem_for_contcts_finite_normal_case} shows that the restrictions of $\npm_U$ and $\npm_V$ to $U\cap V$ coincide.

    We can now define a set function on $\borel$ that will turn out to have the desired properties. If $\mss$ is  Borel set, we set
    \begin{equation}\label{2_eq:patching_together_the_local_measures}
    \npm(\mss)\coloneqq\bigvee\{\npm_U(\mss\cap U) : U\in\rcompacto\}
    \end{equation}
    in $\osext$. The supremum in the right hand side of this equation exists in $\osext$, since the set is upward directed. Indeed, if $U,V\in\rcompacto$, then $\npm_U(\mss\cap U)=\npm_{U\cup V}(\mss\cap U)\leq\npm_{U\cup V}(\mss\cap(U\cup V))$, and likewise $\npm_V(\mss\cap V)\leq\npm_{U\cup V}(\mss\cap(U\cup V))$.

    \emph{We claim that $\npm(\mss)=\npm_{U_0}(\mss)$ whenever $\mss$ is a Borel set and $U_0\in\rcompacto$ is such that $\mss\subseteq U_0$.}

    To see this, note that certainly $\npm(\mss)\geq\npm_{U_0}(\mss\cap U_0)=\npm_{U_0}(\mss)$. Furthermore, if $U\in\rcompacto$, then $\npm_U(\mss\cap U)=\npm_U(\mss\cap U_0\cap U)=\npm_{U_0}(\mss\cap U_0\cap U)\leq\npm_{U_0}(\mss)$. This implies that $\npm(\mss)\leq \npm_{U_0}(\mss)$. We conclude that $\npm(\mss)=\npm_{U_0}(\mss)$, as required.

    \emph{We shall now show that $\npm$ is a measure.}

    It is clear that $\npm(\emptyset)=0$. The first step towards $\sigma$-additivity is to show that $\npm$ is finitely additive. Let $\mss_1$ and $\mss_2$ be disjoint Borel sets. Then
    \begin{align*}
    \npm(\mss_1\cup \mss_2)
    &=\bigvee\{\npm_U((\mss_1\cup \mss_2)\cap U) : U\in\rcompacto\}\\
    &=\bigvee\{\npm_U((\mss_1\cap U)\cup(\mss_2\cap U)) : U\in\rcompacto\}\\
    &=\bigvee\{\npm_U(\mss_1\cap U)+\npm_U(\mss_2\cap U) : U\in\rcompacto\}\\
    &\leq\npm(\mss_1)+\npm(\mss_2)\\
    &=\bigvee\{\npm_U(\mss_1\cap U): U\in\rcompacto\}+\bigvee\{\npm_V(\mss_2\cap V) : V\in\rcompacto\}\\
    &=\bigvee\{\npm_U(\mss_1\cap U)+\npm_V(\mss_2\cap V) : U, V \in\rcompacto\}\\
    &=\bigvee\{\npm_{U\cup V}(\mss_1\cap U)+\npm_{U\cup V}(\mss_2\cap V) : U, V \in\rcompacto\}\\
    &=\bigvee\{\npm_{U\cup V}((\mss_1\cap U)\cup(\mss_2\cap V)) : U, V \in\rcompacto\}\\
    &\leq\bigvee\{\npm_{U\cup V}((\mss_1\cup \mss_2)\cap (U\cup V) : U,V\in\rcompacto\}\\
    &=\bigvee\{\npm_U((\mss_1\cup \mss_2)\cap U) : U\in\rcompacto\}\\
    &=\npm(\mss_1\cup \mss_2).
    \end{align*}
    It follows that $\npm$ is finitely additive.

    Now that we know that $\npm$ is finitely additive, the $\sigma$-additivity is easily seen to be consequence of the fact that $\npm\left(\bigcup_{n=1}^\infty\mss_n\right)=\bigvee_{n=1}^\infty\npm(\mss_n)$ for every increasing sequence $\seq{\mss}$ of Borel sets; we shall now show this. Using \cite[Proposition~4.5]{de_jeu_jiang:2021a}, we see that
    \begin{align*}
    \npm\left( \bigcup_{n=1}^\infty\mss_n\right)
    &=\bigvee\left\{\npm_U\left(\left(\bigcup_{n=1}^\infty\mss_n\right)\cap U\right) : U\in\rcompacto\right\}\\
    &=\bigvee\left\{\bigvee_{n=1}^\infty\npm_U(\mss_n\cap U) : U\in\rcompacto\right\}\\
    &=\bigvee_{n=1}^\infty\left( \bigvee\{\npm_U(\mss_n\cap U) : U\in\rcompacto\}\right)\\
    &=\bigvee_{n=1}^\infty\npm(\mss_n),
    \end{align*}
    as required.

    This concludes the proof that $\npm$ is a measure.

    \emph{The measure $\npm$ is \wir\ at all Borel sets.}

    Since we know that the restrictions of $\npm$ to the $U\in\rcompacto$ are the $\npm_U$, this is clear from \cref{2_eq:patching_together_the_local_measures}.

    \emph{It is now easy to show that $\posmap(f)=\ointm{f}$ for all $f\in\contcts$}.

    Indeed, one can choose $U\in\rcompacto$ such that $f\in\contc{U}$. Since we have already observed that the restriction of $\npm$ to $U$ is $\npm_U$, we see that $\ointm{f}=\orderintegral{U}{f}{\npm_U}$, and $\orderintegral{U}{f}{\npm_U}=\posmap(f)$ by construction.

    \emph{The measure $\npm$ is a Borel measure.}

    Since every compact subset of $\ts$ is contained in a relatively open compact subset of $\ts$, and since  $\npm_U$ is a finite measures for all $U\in\rcompacto$, this is clear.

    \emph{The measure $\npm$ is inner regular at all open sets.}

    Let $V$ be an open subset of $\ts$. Using that, for each $U\in\rcompacto$, $\npm_U$ is inner regular at all open subsets of $U$, we see that
    \begin{align*}
    \npm(V)&=\bigvee\{\npm_U(V\cap U) : U\in\rcompacto\}\\
    &=\bigvee_{U\in\rcompacto}\left\{\bigvee\{\npm_U(K) : K \text{ is compact and } K\subseteq V\cap U \}\right\}\\
    &=\bigvee\{\npm_U(K) :  K \text{ is compact and } K\subseteq V\cap U \text{ for some } U\in\rcompacto\}\\
    &=\bigvee\{\npm(K) : K \text{ is compact and } K\subseteq V\cap U\ \text{ for some } U\in\rcompacto\}\\
    &\leq\bigvee\{\npm(K) :  K \text{ is compact and } K\subseteq V\}\\
    &\leq\npm(V).
    \end{align*}

    \emph{Since $\npm$ is a Borel measure that is inner regular at all open subsets of $\ts$ and \wir\ at all Borel sets, it is a quasi-regular Borel measure.}

    We turn to the remaining statements.

    Suppose that $\widetilde\npm$ is a regular Borel measure that also represents $\posmap$. Then, for all $u\in\rcompacto$, $\widetilde\npm_U$ is a regular Borel measure that represents the restriction of $\posmap$ to $\contc{U}$. By uniqueness, we see that $\widetilde\npm_U=\npm_U$ for all $U\in\rcompacto$. This implies that $\widetilde\npm$ and $\npm$ agree at all compact subsets of $\ts$ and then, by inner regularity, also at all open subsets of $\ts$. Knowing this, and using the outer regularity of $\widetilde\npm$, we see that, for a Borel set $\mss$,
    \begin{align*}
    \widetilde\npm(\mss)
    &=\bigwedge\{\widetilde\npm(V) : V \text{ is open and } \mss\subseteq V\}\\
    &=\bigwedge\{\npm(V) :  V \text{ is open and } \mss\subseteq V\}\\
    &\geq\npm(\mss),
    \end{align*}
    so that $\widetilde\npm\geq\npm$.

    Clearly, if $\npm(\mss)=\largest$, then $\widetilde\npm(\mss)=\largest$ as well.

    Finally, suppose $\mss$ is a Borel set such that $\widetilde\npm(\mss)$ is finite. Using \cref{2_res:mu_is_inner_regular_at_finite_borel_sets} in the second step, we see that
    \begin{align*}
    \npm(\mss)&\leq\widetilde\npm(\mss)\\
    &=\bigvee\{\widetilde\npm(K) :  K \text{ is compact and } K\subseteq \mss\}\\
    &=\bigvee\{\npm(K) :  K \text{ is compact and } K\subseteq \mss\}\\
    &\leq\npm(\mss).
    \end{align*}
    Hence $\npm(\mss)=\widetilde\npm(\mss)$.
\end{proof}

\begin{remark}\label{3_rem:comparison}\quad
	\begin{enumerate}
	\item\label{3_part:comparison_1}
	We now have three main results at our disposal regarding the existence of representing measures for positive operators $\posmap:\contots\to\os$, namely, \cref{2_res:riesz_representation_theorem_for_contcts_normed_case}, \cref{2_res:riesz_representation_theorem_for_contcts_finite_normal_case}, and \cref{2_res:riesz_representation_theorem_for_contcts_normal_case}. The proofs of \cref{2_res:riesz_representation_theorem_for_contcts_normed_case,2_res:riesz_representation_theorem_for_contcts_normal_case} are somewhat long and technical, and it seems appropriate to give examples where they give the optimal result.
	
	We start with an example where \cref{2_res:riesz_representation_theorem_for_contcts_normed_case} is optimal. Take $\ts=\NN$ and $\os=c_0$, and let $\posmap:c_{00}\to c_0$ be the inclusion operator. Then \cref{2_res:riesz_representation_theorem_for_contcts_normed_case} is applicable and yields a regular (infinite) representing measure. \cref{2_res:riesz_representation_theorem_for_contcts_finite_normal_case} is not applicable.  \cref{2_res:riesz_representation_theorem_for_contcts_normal_case} does apply, but does not yield a \emph{regular} representing measure.
	
	To obtain an example where \cref{2_res:riesz_representation_theorem_for_contcts_normal_case} is optimal, take $\ts=\NN$ and $\os=\ell^\infty$, and let $\posmap:c_{00}\to\ell^\infty$ be defined by setting $\posmap(e_n)\coloneqq ne_n$, where $e_n$ is the $n^{\text{th}}$ standard unit vector. Then \cref{2_res:riesz_representation_theorem_for_contcts_normal_case} is applicable, but \cref{2_res:riesz_representation_theorem_for_contcts_normed_case,2_res:riesz_representation_theorem_for_contcts_finite_normal_case} are not.
	
	\item\label{3_part:comparison_2}
	Banach lattices of operators will only rarely satisfy the assumption under~\ref{2_req:riesz_representation_theorem_for_contcts_normed_case_1} in \cref{2_res:riesz_representation_theorem_for_contcts_normed_case}, whereas it is not uncommon to fall within the range of \cref{2_res:riesz_representation_theorem_for_contcts_finite_normal_case,2_res:riesz_representation_theorem_for_contcts_normal_case}.
	\end{enumerate}
\end{remark}

\section{Riesz representation theorems for $\contots$}\label{2_sec:riesz_representation_theorems_for_contots}

\noindent In this section, we establish representation theorems for positive operators $\posmap:\contots\to\os$ that are defined on $\contots$ rather than on $\contcts$. The first one, \cref{2_res:riesz_representation_theorem_for_contots_kb_case}, is based on \cref{2_res:riesz_representation_theorem_for_contcts_normed_case} for normed spaces $\os$. The remaining ones are, although norms will still enter the picture, essentially based on \cref{2_res:riesz_representation_theorem_for_contcts_finite_normal_case} for monotone complete normal spaces $\os$.

The final results of this section, \cref{2_res:riesz_representation_theorem_for_contots_operators_on_kb_space,2_res:riesz_representation_theorem_for_contots_operators_on_hilbert_space}, show that representing measures always exist for positive (not necessarily multiplicative) operators $\posmap$ from $\contots$ into the regular operators on a KB-space and into the self-adjoint operators on a complex Hilbert space, respectively. The positive operator $\posmap$ need not be multiplicative but, as will become apparent, much of what is known for the multiplicative case (see \cite{de_jeu_ruoff:2016} and \cite{conway_A_COURSE_IN_FUNCTIONAL_ANALYSIS_SECOND_EDITION:1990}, respectively) still remains valid, including the fact that the representing measure takes its values in the coinciding bicommutants of $\posmap(\contcts)$ and $\posmap(\contots)$.

Our approach uses automatic continuity to derive results for $\contots$ from those for $\contcts$. Here is an example in the context of normed vector lattices.

\begin{lemma}\label{2_res:measure_also_represents_posmap_on_contots_normed_case}
	 Let $\ts$ be a locally compact Hausdorff space, let $\os$ be a normed vector lattice, and let $\posmap:\contots\to\os$ be a positive operator. Suppose that there exists a measure $\npm:\borel\to\pososext$ such that
	\begin{equation}\label{2_eq:representing_equation_for_contcts_general_normed_case}
	\posmap(f)=\ointm{f}
	\end{equation}
	for all $f$ in $\contcts$ and such that $\contots\subseteq\integrablefunts$.
	
	Then \cref{2_eq:representing_equation_for_contcts_general_normed_case} also holds for all $f\in\contots$.
\end{lemma}

\begin{proof}
	We have supposed that $\contots\subseteq\integrablefunts$, so that it is possible\textemdash this is the point\textemdash to define a positive operator $f\mapsto\ointm{f}$ from $\contots$ into $\os$.
	Since $\contots$ is a Banach lattice and $E$ is a normed vector lattice, this operator is automatically continuous; see \cite[Theorem~4.3]{aliprantis_burkinshaw_POSITIVE_OPERATORS_SPRINGER_REPRINT:2006}, for example. The operator $\posmap$ is likewise automatically continuous. Since they agree on the dense subspace $\contcts$ of $\contots$, they agree on $\contots$.
\end{proof}

Looking at our representation theorems so far, there is indeed one in which normed spaces (possibly normed vector lattices) $\os$ figure, namely, \cref{2_res:riesz_representation_theorem_for_contcts_normed_case}.  \cref{2_res:measure_also_represents_posmap_on_contots_normed_case} makes clear that it is relevant to point out cases in which we know that the representing measure for $\contcts$ from \cref{2_res:riesz_representation_theorem_for_contcts_normed_case} is necessarily finite, i.e., cases where we know that $\{\posmap(f):f\prec \ts\}$ is bounded above in $\os$, once we have the extra information that $\posmap$ is, in fact, the restriction of a positive operator from $\contots$ into $\os$. This is indeed possible, as is shown by \cref{2_res:riesz_representation_theorem_for_contots_kb_case}.

As a preparation, we recall (see \cite[p.~232]{aliprantis_burkinshaw_POSITIVE_OPERATORS_SPRINGER_REPRINT:2006}, for example) that a Banach lattice is a \emph{KB-space} when every increasing norm bounded net in the positive cone is norm convergent. As is well known, the norm limit is then also the supremum of the net. In particular, every increasing norm bounded net in the positive cone of a KB-space is bounded above.

Furthermore, a KB-space satisfies all requirements in \cref{2_res:riesz_representation_theorem_for_contcts_normed_case}.
Indeed, it has an order continuous norm (see  \cite[Theorem~7.1]{wnuk_BANACH_LATTICES_WITH_ORDER_CONTINUOUS_NORMS:1999}, for example), so that it is \Dc\ (see \cite[Corollary~4.10]{aliprantis_burkinshaw_POSITIVE_OPERATORS_SPRINGER_REPRINT:2006}, for example). The fact that is a Banach space, together with the fact that a norm limit of an increasing sequence in a Banach lattice is also its supremum, shows that the second requirement is fulfilled. The third follows immediately from the normality of every Banach lattice with an order continuous norm (see \cite[Proposition~3.10]{de_jeu_jiang:2021a}, for example).

After these observations, the next result is easily proved.

\begin{theorem}\label{2_res:riesz_representation_theorem_for_contots_kb_case}
Let $\ts$ be a locally compact Hausdorff space, let $\os$ be a KB-space, and let $\posmap:\contots\to\os$ be a positive operator.

Then there exists a unique regular Borel measure $\npm:\borel\to\pososext$ on the Borel $\sigma$-algebra $\borel$ of $\ts$ such that
\begin{equation}\label{2_eq:representing_measure_for_contots_kb_case}
\posmap(f)=\ointm{f}
\end{equation}
for all $f\in\contcts$. The measure $\npm$ is finite, and \cref{2_eq:representing_measure_for_contots_kb_case} also holds for all $f\in\contots$. If $V$ is a non-empty open subset of $\ts$, then
\begin{equation*}
\npm(V)=\bigvee\{\posmap(f) : f\prec V\}
\end{equation*}
in $\os$. If $K$ is a compact subset of $\ts$, then
\begin{equation*}
\npm(K)=\bigwedge\{\posmap(f) : K\prec f\}
\end{equation*}
in $\os$.

\end{theorem}

\begin{proof}
As observed preceding the theorem, \cref{2_res:riesz_representation_theorem_for_contcts_normed_case} can be applied. The measure $\npm$ that is obtained satisfies $\npm(\ts)=\{\posmap(f): f\in\pos{\contcts},\,\norm{f}\leq 1\}$. Since $\posmap$ is positive, it is continuous, so that the set $\{\posmap(f): f\in\pos{\contcts},\,\norm{f}\leq 1\}$ is norm bounded. Since it is also upward directed as a consequence of the positivity of $\posmap$, the remarks preceding the theorem show that it is bounded above. We conclude that $\npm$ is a finite measure, and then \cref{2_res:measure_also_represents_posmap_on_contots_normed_case} applies to show that \cref{2_eq:representing_measure_for_contots_kb_case} is valid for all $f\in\contots$.
\end{proof}

We now turn to normal spaces, where we have \cref{2_res:riesz_representation_theorem_for_contcts_finite_normal_case} to start from. As we shall see, we shall actually obtain a generalisation of \cref{2_res:riesz_representation_theorem_for_contots_kb_case}. For these spaces, we need the following companion result of \cref{2_res:measure_also_represents_posmap_on_contots_normed_case}.

\begin{lemma}\label{2_res:measure_also_represents_posmap_on_contots_normal_case}
	Let $\ts$ be a locally compact Hausdorff space, let $\os$ be a \smc\ partially ordered vector space such that the positive functionals separate the points of $\os$, and let $\posmap: \contots\to\os$ be a positive operator. Suppose that there exists a measure $\npm:\borel\to\pososext$ such that
	\begin{equation}\label{2_eq:representing_equation_for_contcts_general_normal_case}
	\posmap(f)=\ointm{f}
	\end{equation}
	for all $f$ in $\contcts$ and such that $\contots\subseteq\integrablefunts$.
	
	Then \cref{2_eq:representing_equation_for_contcts_general_normal_case} also holds for all $f\in\contots$.
\end{lemma}

\begin{proof}
	Take a positive functional $x^\prime$ on $\os$. Then the maps $f\mapsto\f{\posmap(f),x^\prime}$ and $f\mapsto\f{\ointm{f},x^\prime}$ are both positive functionals on the Banach lattice $\contots$. They agree on the dense subspace $\contcts$ of $\contots$, so they are equal by continuity. Since the positive functionals on $\os$ separate the points, it follows that $\posmap(f)=\ointm{f}$ for all $f\in\contots$.
\end{proof}

Analogously to case above for normed spaces, it is, in view of \cref{2_res:measure_also_represents_posmap_on_contots_normal_case}, relevant to point out cases where we know that the representing measure for $\contcts$ in \cref{2_res:riesz_representation_theorem_for_contcts_finite_normal_case} is necessarily finite, i.e., that $\{\posmap(f):f\prec \ts\}$ is bounded above, once we have the extra information that $\posmap$ is the restriction of a positive operator defined on $\contots$.

This is indeed possible again for a class of partially ordered vector spaces that we now introduce.

\begin{definition}\label{2_def:quasi_perfect_spaces}
	A partially ordered vector pace $\os$ is \emph{quasi-perfect} when the two following conditions are both satisfied:
	\begin{enumerate}
		\item\label{2_part:quasi_perfect_spaces_1}
		$\os$ is normal;
		\item\label{2_part:quasi_perfect_spaces_2}
		whenever an increasing net $\net{x}$ in $\posos$ is such that $\sup\f{x_\lambda,x^\prime}<\infty$ for each $x^\prime\in\pos{(\odualos)}$, then this net has a supremum in $\os$.
	\end{enumerate}
\end{definition}

To motivate the terminology, let us recall that a vector lattice is called \emph{perfect} if the natural Riesz homomorphism from $\os$ into $\ocdual{(\ocdualos)}$ is a surjective isomorphism. The following alternate characterisation is due to Nakano; see \cite[Theorem~1.71]{aliprantis_burkinshaw_POSITIVE_OPERATORS_SPRINGER_REPRINT:2006}.

\begin{theorem}\label{2_res:nakano}
	A vector lattice $\os$ is a perfect vector lattice if and only if	the following two conditions hold:
	\begin{enumerate}
		\item\label{2_part:nakano_1}
		$\os$ is normal;
		\item\label{2_part:nakano_2}
		whenever an increasing net $\net{x}$ in $\posos$ is such that $\sup\f{x_\lambda,x^\prime}<\infty$ for each $x^\prime\in\pos{(\ocdualos)}$, then this net has a supremum in $\os$.
	\end{enumerate}
\end{theorem}

We see from this that every perfect vector lattice satisfies the more lenient conditions in \cref{2_def:quasi_perfect_spaces}; hence the name `quasi-perfect'.

Clearly, a quasi-perfect partially ordered vector space is \mc.

We shall collect a number of examples of quasi-perfect partially ordered vector spaces in \cref{2_res:examples_of_quasi_perfect_spaces}. As a preparation, we recall that the norm on a Banach lattice is said to be a \emph{Levi norm} if every norm bounded increasing net in the positive cone has a supremum. It is not difficult to see that the KB-spaces are precisely the Banach lattices with a Levi norm that is order continuous.

\begin{lemma}\label{2_res:inclusions_between_classes_of_riesz_spaces}
The following inclusions between classes of vector lattices hold.
	\begin{enumerate}
		\item\label{2_part:inclusions_between_classes_of_riesz_spaces_1}
		Every KB-space is a normal Banach lattice with a Levi norm;
		\item\label{2_part:inclusions_between_classes_of_riesz_spaces_2}
		Every normal Banach lattice with a Levi norm is a quasi-perfect vector lattice;
		\item\label{2_part:inclusions_between_classes_of_riesz_spaces_3}
		Every KB-space is a perfect vector lattice;
		\item\label{2_part:inclusions_between_classes_of_riesz_spaces_4}
		Every perfect vector lattice is a quasi-perfect vector lattice.
	\end{enumerate}
\end{lemma}

\begin{proof}
We prove part~\ref{2_part:inclusions_between_classes_of_riesz_spaces_1}. As noted earlier, a KB-space has an order continuous norm, so that it is normal. If $\net{x}$ is a norm bounded increasing net in the positive cone of a KB-space, then it is norm convergent. As mentioned earlier, the norm limit of the net is then also the supremum of the net.

We prove part~\ref{2_part:inclusions_between_classes_of_riesz_spaces_2}. Suppose that $\net{x}$ is an increasing net in the positive cone of a normal Banach lattice $\os$ with a Levi norm such that $\sup\f{x_\lambda,x^\prime}<\infty$ for each $x^\prime\in\pos{(\odualos)}$. Then $\{\f{x_\lambda,x^\prime} : \lambda\in\Lambda\}$ is bounded in $\RR$ for every $x^\prime\in\odualos=\ndualos$. We conclude that $\net{x}$ is norm bounded. Since the norm is supposed to be a Levi norm, the increasing net $\net{x}$ has a supremum in $\os$.

We prove part~\ref{2_part:inclusions_between_classes_of_riesz_spaces_3}. As noted above, a KB-space is normal. Suppose that $\net{x}$ is an increasing net in the positive cone of a KB-space $\os$ such that $\sup\f{x_\lambda,x^\prime}<\infty$ for each $x^\prime\in\pos{(\ocdualos)}$. This implies that the set $\{\f{x_\lambda,x^\prime} : \lambda\in\Lambda\}$ is bounded in $\RR$ for every $x^\prime\in\ocdualos=\ndualos$, where the latter equality holds because $\os$ has an order continuous norm. We conclude that $\net{x}$ is norm bounded. Since the space is a KB-space, the increasing net $\net{x}$ has a norm limit in $\os$, and this norm limit is then also the supremum of the net.

Part~\ref{2_part:inclusions_between_classes_of_riesz_spaces_4} is clear.
\end{proof}

\begin{proposition}\label{2_res:examples_of_quasi_perfect_spaces}
	The following spaces are quasi-perfect partially ordered vector spaces:
	\begin{enumerate}
		\item\label{2_part:examples_of_quasi_perfect_spaces_1}
		perfect Riesz spaces;
		\item\label{2_part:examples_of_quasi_perfect_spaces_2}
		normal Banach lattices with a Levi norm, such as KB-spaces and, still more in particular, reflexive Banach lattices;
		\item\label{2_part:examples_of_quasi_perfect_spaces_3}
		for \SOT-closed complex linear subspaces $L$ of $\boundedh$, where $\hilbert$ is a complex Hilbert space: the real vector spaces $L_\sa$ consisting of all self-adjoint elements of $L$;
		\item\label{2_part:examples_of_quasi_perfect_spaces_4} JBW algebras.
	\end{enumerate}
\end{proposition}

\begin{proof}
	The parts~\ref{2_part:examples_of_quasi_perfect_spaces_1} and~\ref{2_part:examples_of_quasi_perfect_spaces_2} are immediate from \cref{2_res:examples_of_quasi_perfect_spaces}.
	
	We prove part~\ref{2_part:examples_of_quasi_perfect_spaces_3}. Suppose that $\net{T}$ is an increasing net in $\pos{L_\sa}$ such that $\sup\,(T_\lambda,\varphi)<\infty$ for all $\varphi\in\ocdual{(L_\sa)}$. Then, in particular, $\sup\,\lrinp{T_\lambda x,x}<\infty$ for all $x\in\hilbert$. Polarisation yields that $\sup\,\abs{\lrinp{T_\lambda x,y}}<\infty$ for all $x,y\in\hilbert$. It follows from this that the net is uniformly bounded,  which implies that there exists an $M\geq 0$ such that $T_\lambda\leq M\idop$ for all $\lambda$. According to \cite[Lemma~I.6.4]{davidson_C-STAR-ALGEBRAS_BY_EXAMPLE:1996}, the \SOT-limit of the net exists in $\boundedh$, and this limit is the supremum of the net in $\boundedh$. Since $L$ is \SOT-closed, this limit is in $L_\sa$; it is then the supremum of the net in $L_\sa$.
	
	We prove part~\ref{2_part:examples_of_quasi_perfect_spaces_4}. The normality a JBW-algebra is stated in \cite[Theorem~2.17]{alfsen_shultz_GEOMETRY_OF_STATE_SPACES_OF_OPERATOR_ALGEBRAS:2003}. Suppose that $\net{x}$ is an increasing net in the positive cone $\pos{\jbw}$ of a JBW-algebra $\jbw$ with the property that $\sup(x_\lambda,x^\prime)<\infty$ for each $x^\prime\in\pos{(\odual{\jbw})}$. We know from \cite[Theorem~1.11]{alfsen_shultz_GEOMETRY_OF_STATE_SPACES_OF_OPERATOR_ALGEBRAS:2003} that $\jbw$ is a norm complete order unit space, and then its norm and order dual coincide by \cite[Theorem~1.19]{alfsen_shultz_STATE_SPACES_OF_OPERATOR_ALGEBRAS:2001}. The uniform boundedness principle then implies that the net is norm bounded, so that it is also order bounded in the order unit space $\jbw$. Since JBW-algebras are \mc\ by definition, the net $\net{x}$ has a supremum in $\jbw$.
\end{proof}

With the examples of quasi-perfect partially ordered vector spaces in \cref{2_res:examples_of_quasi_perfect_spaces} in mind, we
now establish a Riesz representation theorem for such spaces. It is based on \cref{2_res:riesz_representation_theorem_for_contcts_finite_normal_case} for normal spaces and includes \cref{2_res:riesz_representation_theorem_for_contots_kb_case} for KB-spaces, the proof of which was based on \cref{2_res:riesz_representation_theorem_for_contcts_normed_case} for normed spaces, as a special case.

\begin{theorem}\label{2_res:riesz_representation_theorem_for_contots_normal_case}
	Let $\ts$ be a locally compact Hausdorff space, let $\os$ be a quasi-perfect partially ordered vector space, and let $\posmap:\contots\to\os$ be a positive operator.
	
	Then there exists a unique regular Borel measure $\npm:\borel\to\pososext$ on the Borel $\sigma$-algebra $\borel$ of $\ts$ such that
	\begin{equation}\label{2_eq:representing_measure_for_contots_normal_case}
	\posmap(f)=\ointm{f}
	\end{equation}
	for all $f\in\contcts$. The measure $\npm$ is finite, and \cref{2_eq:representing_measure_for_contots_normal_case} also holds for all $f\in\contots$. If $V$ is a non-empty open subset of $\ts$, then
	\begin{equation*}
	\npm(V)=\bigvee\{\posmap(f) : f\prec V\}
	\end{equation*}
	in $\os$. If $K$ is a compact subset of $\ts$, then
	\begin{equation*}
	\npm(K)=\bigwedge\{\posmap(f) : K\prec f\}
	\end{equation*}
	in $\os$.
\end{theorem}

\begin{proof}
We claim that $\{\posmap(f): f\prec \ts\}$ is bounded above in $\os$. To see this, take an $x^\prime$ in $\pos{(\odualos)}$, and consider the map $f\mapsto\f{\posmap(f),x^\prime}$ from $\contots$ to $\RR$. The positivity of this functional implies that it is continuous. It follows from this continuity that $\{(\posmap(f),x^\prime): f\prec\ts\}$ is bounded in $\RR$. Since the set $\{\posmap(f): f\prec \ts\}$ is upward directed, the fact that $\os$ is quasi-perfect now implies that this set has a supremum in $\os$. Consequently, it is bounded above.

Now that we know this, \cref{2_res:riesz_representation_theorem_for_contcts_finite_normal_case} applies, and then an appeal to \cref{2_res:measure_also_represents_posmap_on_contots_normal_case} completes the proof.
\end{proof}

\cref{2_res:riesz_representation_theorem_for_contots_normal_case}  has a consequence that will become important later on when considering representations of $\contots$ on KB-spaces in \cite{de_jeu_jiang:2021c}. We need a preparatory result that seems worth recording explicitly.

\begin{lemma}\label{2_res:regular_operators_on_kb_space_have_levi_norm}
	Let $\os$ be a KB-space. Then the regular norm on $\regularop{\os}$ is a Levi norm.
\end{lemma}

\begin{proof}
Suppose that $0\leq\net{T}\uparrow$ in $\regularop{\os}$ is a net that is bounded in the regular norm, which coincides with the operator norm on the positive operators. Choose $M\geq 0$ such that $\norm{T_\lambda}\leq M$ for all $\lambda\in\Lambda$. If $x\in\posos$, then $\{T_\lambda x\}_{\lambda\in\Lambda}$ is an increasing net in $\posos$. Furthermore, we have $\norm{T_\lambda x}\leq M\norm{x}$ for all $\lambda\in\Lambda$. Since $\os$ is a KB-space, the supremum of the increasing norm bounded positive net $\{T_\lambda x\}_{\lambda\in\Lambda}$ exists in $\os$. Hence the net $\net{T}$ has a supremum in $\regularop{\os}$ by  \cite[Theorem~1.19]{aliprantis_burkinshaw_POSITIVE_OPERATORS_SPRINGER_REPRINT:2006} or by the more general \cite[Proposition~3.1]{de_jeu_jiang:2021a}, for example.
\end{proof}

We can now establish a Riesz representation theorem for positive (not necessarily multiplicative) operators from $\contots$ into the regular operators on KB-spaces.

\begin{theorem}\label{2_res:riesz_representation_theorem_for_contots_operators_on_kb_space}
	Let $\ts$ be a locally compact Hausdorff space, let $\os$ be a KB-space, and let $\posmap:\contots\to\regularop{\os}$ be a positive \uppars{not necessarily multiplicative} operator.
	
	Then there exists a unique regular Borel measure $\npm:\borel\to \overline{\pos{\regularop{\os}} }$ on the Borel $\sigma$-algebra $\borel$ of $\ts$ such that
	\begin{equation}\label{2_eq:riesz_representation_theorem_for_contots_operators_on_kb_space_1}
	\posmap(f)=\ointm{f}
	\end{equation}
	for all $f\in\contcts$. The measure $\npm$ is finite, and \cref{2_eq:riesz_representation_theorem_for_contots_operators_on_kb_space_1} also holds for all $f\in\contots$. If $V$ is a non-empty open subset of $\ts$, then
	\begin{equation}\label{2_eq:riesz_representation_theorem_for_contots_operators_on_kb_space_2}
	\npm(V)=\bigvee\{\posmap(f) : f\prec V\}
	\end{equation}
	in $\os$. If $K$ is a compact subset of $\ts$, then
	\begin{equation*}
	\npm(K)=\bigwedge\{\posmap(f) : K\prec f\}
	\end{equation*}
	in $\os$. Suppose that $\seq{\mss}$ is a pairwise disjoint sequence in $\borel$. Then, for $x\in\os$,  $\npm\left(\bigcup_{n=1}^\infty\mss_n\right)x=\sum_{n=1}^\infty \npm(\mss_n) x$ in the norm topology of $\os$.
	
	Define the extension $\posmap:\integrablefunts\to\regularop{\os}$ via \cref{2_eq:riesz_representation_theorem_for_contots_operators_on_kb_space_1}. For $x\in\os$, $x^\prime\in\odualos$, and $\mss\in\borel$, set $\npm_{x,x^\prime}(\mss)\coloneqq(\npm(\mss)x,x^\prime)$. Then $\npm_{x,x^\prime}:\borel\to\posR$ is a regular Borel measure on $\ts$ and, for $f\in\integrablefunts$, $\posmap (f)$ is the unique element of $\regularop{\os}$ such that
	\[
	(\posmap(f)x,x^\prime)=\int_\ts\! f\di{\npm_{x,x^\prime}}
	\]
	for all $x\in\os$ and $x^\prime\in\odualos$.

	Furthermore, the \SOT-closed linear subspaces of the bounded \uppars{not necessarily regular} operators on $\os$ that are generated by the following sets are equal:
	\begin{enumerate}
		\item\label{2_part:riesz_representation_theorem_for_contots_operators_on_kb_space_1}
		$\{\posmap(f): f\in\contcts\}$
		\item\label{2_part:riesz_representation_theorem_for_contots_operators_on_kb_space_2}
		$\{\posmap(f): f\in\contots\}$
		\item\label{2_part:riesz_representation_theorem_for_contots_operators_on_kb_space_3}
		$\{\posmap(f): f\in\boundedmeasfunts\}$
		\item\label{2_part:riesz_representation_theorem_for_contots_operators_on_kb_space_4}
		$\{\posmap(f): f\in\integrablefunts\}$
		\item\label{2_part:riesz_representation_theorem_for_contots_operators_on_kb_space_5}
		$\{\npm(V): V\text{ is an open subset of }\ts\}$;
		\item\label{2_part:riesz_representation_theorem_for_contots_operators_on_kb_space_6}
		$\{\npm(K): K\text{ is a compact subset of }\ts\}$;
		\item\label{2_part:riesz_representation_theorem_for_contots_operators_on_kb_space_7}
		$\{\npm(\mss): \mss\text{ is a Borel subset of }\ts\}$.
\end{enumerate}
	
\end{theorem}

\begin{proof}
	Since $\os$ is a KB-space, part~(1) of
	\cite[Theorem~3.14]{de_jeu_jiang:2021a} shows that the order continuity of the norm on $\os$ implies that $\regularop{\os}$ is a normal partially ordered vector space. Since we have established in \cref{2_res:regular_operators_on_kb_space_have_levi_norm} that the regular norm on $\regularop{\os}$ is a Levi norm, \cref{2_res:riesz_representation_theorem_for_contots_normal_case} applies to the partially ordered vector space $\regularop{\os}$. On also taking part~\ref{2_part:comparison_with_vector_measures_and_sot_sigma_additivity_2} of \cref{2_rem:comparison_with_vector_measures_and_sot_sigma_additivity} and \cite[Proposition~6.8]{de_jeu_jiang:2021a} into account, all statements in the theorem follow, except the one on the equality of the \SOT-closed linear subspaces. We shall now establish this. Let $L_1,\dotsc, L_7$ denote the \SOT-closed linear subspaces of the bounded operators on $\os$ that are generated by the sets under \ref{2_part:riesz_representation_theorem_for_contots_operators_on_kb_space_1}, $\dotsc$,\ref{2_part:riesz_representation_theorem_for_contots_operators_on_kb_space_7}, respectively.
	
	We show that $L_1\supseteq L_5$. Suppose that $V$ is an open subset of $\ts$. Then \cref{2_eq:riesz_representation_theorem_for_contots_operators_on_kb_space_2} shows that $\posmap(f)x\uparrow \npm(V)x$ for all $x\in\posos$, where the supremum is the one of the increasing net $\{\posmap(f)x\in\contcts: f\prec V\}$. Since the norm of $\os$ is order continuous, it now follows that $\npm(V)$ is the strong operator limit of the $\posmap(f)$ for $f\prec V$. Hence $\npm(V)\in L_1$. It follows that $L_1\supseteq L_5$.
	
	Using the regularity of $\npm$, a similar argument shows that $L_5=L_6=L_7$.
	
	Using the definition of the order integral, an again similar argument shows that $L_7\supseteq L_4$. Since it is trivial that $L_4\supseteq L_3\supseteq L_2\supseteq L_1$, the proof is complete.
\end{proof}	

\begin{remark}\label{2_rem:bicommutant_kb_space}
It follows from \cref{2_res:riesz_representation_theorem_for_contots_kb_case} that the commutants and then also the bicommutants (both in the bounded operators on $\os$) of the seven sets in \cref{2_res:riesz_representation_theorem_for_contots_kb_case} are equal. Consequently, $\npm$ takes its values in the coinciding bicommutants of these sets.
\end{remark}

We conclude this section with a Riesz representation theorem for positive (not necessarily multiplicative) operators from $\contots$ into the self-adjoint operators on a complex Hilbert space.

\begin{theorem}\label{2_res:riesz_representation_theorem_for_contots_operators_on_hilbert_space}
	Let $\ts$ be a locally compact Hausdorff space, let $\hilbert$ be a complex Hilbert space, and let $\posmap:\contots\to \boundedh_\sa$ be a positive \uppars{not necessarily multiplicative} operator.
	Let $L$ be a strongly closed complex linear subspace of $\boundedh$ that contains $\posmap(\contots)$, and let $L_\sa$ denote be the real vector space that consists of the self-adjoint operators in $L$, supplied with the partial ordering that is inherited from the usual partial ordering on $\boundedh_\sa$.
	
	Then there exists a unique regular Borel measure $\npm:\borel\to \overline{{}\pos{L_\sa}}$ on the Borel $\sigma$-algebra $\borel$ of $\ts$ such that
	\begin{equation}\label{2_eq:riesz_representation_theorem_for_contots_operators_on_hilbert_space}
	\posmap(f)=\ointm{f}
	\end{equation}
	for all $f\in\contcts$. The measure $\npm$ is finite, and \cref{2_eq:riesz_representation_theorem_for_contots_operators_on_hilbert_space} also holds for all $f\in\contots$.	If $V$ is a non-empty open subset of $\ts$, then
	\begin{equation*}
	\npm(V)=\bigvee\{\posmap(f) : f\prec V\}
	\end{equation*}
	in  $L_\sa$. If $K$ is a compact subset of $\ts$, then
	\begin{equation*}
	\npm(K)=\bigwedge\{\posmap(f) : K\prec f\}
	\end{equation*}
	in  $L_\sa$. Suppose that $\seq{\mss}$ is a pairwise disjoint sequence in $\borel$. Then, for $x\in\hilbert$,  $\npm\left(\bigcup_{n=1}^\infty\mss_n\right)x=\sum_{n=1}^\infty \npm(\mss_n) x$ in the norm topology of $\hilbert$.

	Define the extension $\posmap:\integrablefunts\to L_\sa$ via \cref{2_eq:riesz_representation_theorem_for_contots_operators_on_kb_space_1}. For $x\in\hilbert$ and $\mss\in\borel$, set $\npm_x(\mss)\coloneqq\lrinp{\npm(\mss)x,x}$. Then $\npm_{x}:\borel\to\posR$ is a regular Borel measure on $\ts$ and, for $f\in\integrablefunts$, $\posmap (f)$ is the unique element of $\boundedh$ such that
	\[
	\lrinp{\posmap(f)x,x}=\int_\ts\! f\di{\npm_{x}}
	\]
	for all $x\in\hilbert$.

		Furthermore, the \SOT-closed real linear subspaces of $\boundedh$ that are generated by the following sets are equal:
	\begin{enumerate}
		\item\label{2_part:riesz_representation_theorem_for_contots_operators_on_hilbert_space_1}
		$\{\posmap(f): f\in\contcts\}$
		\item\label{2_part:riesz_representation_theorem_for_contots_operators_on_hilbert_space_2}
		$\{\posmap(f): f\in\contots\}$
		\item\label{2_part:riesz_representation_theorem_for_contots_operators_on_hilbert_space_3}
		$\{\posmap(f): f\in\boundedmeasfunts\}$
		\item\label{2_part:riesz_representation_theorem_for_contots_operators_on_hilbert_space_4}
		$\{\posmap(f): f\in\integrablefunts\}$
		\item\label{2_part:riesz_representation_theorem_for_contots_operators_on_hilbert_space_5}
		$\{\npm(V): V\text{ is an open subset of }\ts\}$;
		\item\label{2_part:riesz_representation_theorem_for_contots_operators_on_hilbert_space_6}
		$\{\npm(K): K\text{ is a compact subset of }\ts\}$;
		\item\label{2_part:riesz_representation_theorem_for_contots_operators_on_hilbert_space_7}
		$\{\npm(\mss): \mss\text{ is a Borel subset of }\ts\}$.
		\end{enumerate}
	The analogous statement for the equality of seven \SOT-closed complex linear subspaces of $\boundedh$ is also true.
\end{theorem}

\begin{proof}
	In view of part~\ref{2_part:examples_of_quasi_perfect_spaces_3} of \cref{2_res:examples_of_quasi_perfect_spaces},  \cref{2_res:riesz_representation_theorem_for_contcts_finite_normal_case} applies to the positive operator $\posmap:\contots\to L_{\sa}$. Together with part~\ref{2_part:comparison_with_vector_measures_and_sot_sigma_additivity_2} of \cref{2_rem:comparison_with_vector_measures_and_sot_sigma_additivity} and \cite[Proposition~6.8]{de_jeu_jiang:2021a}, it shows that a measure $\npm$ exists that has all the properties in the theorem, except that we still need to show that \cref{2_eq:riesz_representation_theorem_for_contots_operators_on_hilbert_space} holds for all $f\in\contots$, and that the statement on the equality of seven \SOT-closed complex linear subspaces of $\boundedh$ is valid.
		
		For the validity of \cref{2_eq:riesz_representation_theorem_for_contots_operators_on_hilbert_space}, we consider, for $x\in\hilbert$, the functionals on $\contots$ that are defined by $f\mapsto\inp{\posmap(f)x,x}$ and $f\mapsto\lrinp{\left(\ointm{f}\right)\cdot x,x,}$. Being positive, they are continuous. Since they agree on $\contcts$, they are equal.
		Hence $\inp{\posmap(f)x,x}=\lrinp{\left(\ointm{f}\right)\cdot x,x}$ for all $f\in\contots$ and $x\in \hilbert$.
		This implies that \cref{2_eq:riesz_representation_theorem_for_contots_operators_on_hilbert_space} holds for all $f\in\contots$.
		
	For a monotone bounded net of self-adjoint operators on $\hilbert$, its supremum (or infimum) is its limit in the strong operator topology. Using this as in the proof of \cref{2_res:riesz_representation_theorem_for_contots_operators_on_kb_space}, the equality of the \SOT-closed real linear subspaces that are generated by the seven sets follows easily from the regularity of $\npm$ and the definition of the order integral. The analogous statement for the complex linear subspaces follows from the real case.
\end{proof}

\begin{remark}\label{2_rem:bicommutant_hilbert_space}
	As in the case of \cref{2_res:riesz_representation_theorem_for_contots_operators_on_kb_space}, it follows from \cref{2_res:riesz_representation_theorem_for_contots_operators_on_hilbert_space} that the commutants and then also the bicommutants in $\boundedh$ of the seven sets in \cref{2_res:riesz_representation_theorem_for_contots_operators_on_hilbert_space} are equal. Consequently, $\npm$ takes its values in these coinciding bicommutants.
\end{remark}

\section*{Acknowledgements} The authors thank Marten Wortel for helpful discussions on JBW-algebras.

\bibliographystyle{plain}
\urlstyle{same}

\bibliography{general_bibliography}

\end{document}